\documentclass[a4paper, 10pt]{article}

\usepackage{amsthm,mathtools,amssymb,bbold,subcaption,natbib}
\usepackage[section]{placeins}
\usepackage[shortlabels]{enumitem}
\RequirePackage[colorlinks,citecolor=blue,urlcolor=blue]{hyperref}
\usepackage{geometry}
\newtheorem{Lem}{Lemma}[section]
\newtheorem{As}[Lem]{Assumption}
\newtheorem{Theorem}[Lem]{Theorem}

\newtheorem{Def}[Lem]{Definition}
\newtheorem{Notation}[Lem]{Notation}
\newtheorem{Cor}[Lem]{Corollary}
\newtheorem{Conj}[Lem]{Conjecture}

\newcommand{\grad}{\mbox{\rm grad}}

\newcommand{\vol}{\mbox{\rm vol}}

\newcommand{\Prb}{\mathbb P}

\newcommand{\Hess}{\mbox{\rm Hess\,}} 

\newcommand{\iid}{\operatorname{\stackrel{i.i.d.}{\sim}}}

\newcommand{\wF}{\widetilde{F}}
\newcommand{\eps}{\varepsilon}

\newcommand{\oli}{\overline}
\newcommand{\cut}{\textnormal{Cut}}

\makeatletter
\renewcommand{\paragraph}{%
  \@startsection{paragraph}{4}%
  {\z@}{3.25ex \@plus 1ex \@minus .2ex}{-0.5em}%
  {\normalfont\normalsize\bfseries}%
}
\renewcommand{\@listI}{%
  \leftmargin=25pt
  \rightmargin=0pt
  \labelsep=5pt
  \labelwidth=20pt
  \itemindent=0pt
  \listparindent=0pt
  \topsep=5pt plus 2pt minus 4pt
  \partopsep=0pt plus 1pt minus 1pt
  \parsep=1pt plus 1pt
  \itemsep=\parsep
}
\makeatother

\begin{document}
\title{Geometrical Smeariness -- A new Phenomenon of Fr\'echet Means}
\author{Benjamin Eltzner\footnote{Felix-Bernstein-Institut f\"ur Mathematische Statistik in den Biowissenschaften, Georg-August-Universit\"at G\"ottingen}} 
\maketitle

\begin{abstract}
  In the past decades, the central limit theorem (CLT) has been generalized to non-Euclidean data spaces. Some years ago, it was found that for some random variables on the circle, the sample Fr\'echet mean fluctuates around the population mean asymptotically at a scale $n^{-\tau}$ with exponent $\tau < 1/2$ with a non-normal distribution if the probability density at the antipodal point of the mean is $\frac{1}{2\pi}$. The author and his collaborator recently discovered that $\tau = 1/6$ for some random variables on higher dimensional spheres. In this article we show that, even more surprisingly, the phenomenon on spheres of higher dimension is qualitatively different from that on the circle, as it depends purely on geometrical properties of the space, namely its curvature, and not on the density at the antipodal point. This gives rise to the new concept of geometrical smeariness. In consequence, the sphere can be deformed, say, by removing a neighborhood of the antipodal point of the mean and gluing a flat space there, with a smooth transition piece. This yields smeariness on a manifold, which is diffeomorphic to Euclidean space. We give an example family of random variables with 2-smeary mean, i.e. with $\tau = 1/6$, whose range has a hole containing the cut locus of the mean. The hole size exhibits a curse of dimensionality as it can increase with dimension, converging to the whole hemisphere opposite a local Fr\'echet mean. We observe smeariness in simulated landmark shapes on Kendall pre-shape space and in real data of geomagnetic north pole positions on the two-dimensional sphere.
\end{abstract}

%\tableofcontents

\section{Introduction}

The central limit theorem is a cornerstone of statistics. Building on this fundamental theorem for real random variables, asymptotic theory has been developed to encompass random variables in a wide variety of data spaces including vector spaces (presented in many textbooks, e.g. \cite{MKB79}) and spaces like manifolds, e.g. \cite{BP03,BP05,BB12}, and stratified spaces, e.g. \cite{BLO13,HHMMN13,HMMN15,BLO18}. The last decades have especially seen the development of asymptotic theory for the Fr\'echet mean (also called \textit{barycenter}) and also for more general data descriptors on non-Euclidean data spaces \cite{BP03,BP05,BB08,H11a,H11b,BB12}. Determining necessary and sufficient conditions for standard asymptotic rates of the mean on non-Euclidean spaces is an ongoing endeavor considered by several recent publications, e.g. \cite{BL17,Schoetz2019,ACGP19,GPRS19,EGHT19}.

The seminal work by \cite{Sturm03} showed that the Fr\'echet mean is unique on metric spaces which are non-positively curved in the sense of Alexandrov. \cite[Theorem 2.4.1]{Afsari09} showed that in simply connected spaces of non-positive curvature the Hessian of the squared geodesic distance, i.e. the squared length of a shortest geodesic between two points, is strictly positive definite. This leads to a CLT with rate $n^{-1/2}$ under some technical assumptions which are fairly straightforward in finite dimension. On the other hand, research into asymptotics on positively curved spaces has lead to the discovery of the phenomenon called ``smeariness'' by \cite{HH15}, where the asymptotic rate of the mean on the circle is lower than $n^{-1/2}$. It is clear that this is an obstacle to hypothesis testing, firstly because table quantiles based on asymptotic considerations, as in the $T^2$-test, cannot be used and secondly because much larger sample sizes are required to improve the power of hypothesis tests.

Lower rates of convergence than $n^{-1/2}$ are known for many estimators. A simple example is the center of an interval of fixed length containing the largest possible fraction of data points, as described by \cite{vdV00} and some recent examples include \cite{LY12,LM15,CC15}. These rates are usually independent of the random variables within a given class and are a property of the estimated descriptor and the data and descriptor spaces.

Smeariness of the mean, however, depends on the random variable: the asymptotic rate differs for different random variables, as shown for the circle by \cite{HH15,Hun17} and for the sphere by \cite{EH19}. The dependence on the random variable exacerbates the problem of hypothesis testing, since it is not known in advance for a certain data set if smeariness may play a complicating role or not.

Smeariness on the circle occurs, assuming a probability density in a neighborhood of the antipodal, if and only if the value of the probability density at the antipodal point of the mean corresponds to that of the uniform distribution, i.e. $\frac{1}{2\pi}$ if the circle is parametrized in arc length. A general framework for such phenomena has been derived by \cite{EH19} using empirical process theory. The examples for arbitrary-dimensional spheres given by \cite{EH19} also feature a non-zero probability density at the antipodal point of the mean, but it was not shown whether this feature is necessary for smeariness. Here we identify fundamental differences between smeariness on the circle and on spheres of dimension $m \ge 2$. The difference can be traced back to the question whether geodesics can circumvent the cut locus. If they can, we call this \emph{geometrical smeariness}, otherwise \emph{cut locus smeariness}.

\begin{figure}[h!]
  \centering
  \includegraphics[width=0.4\textwidth]{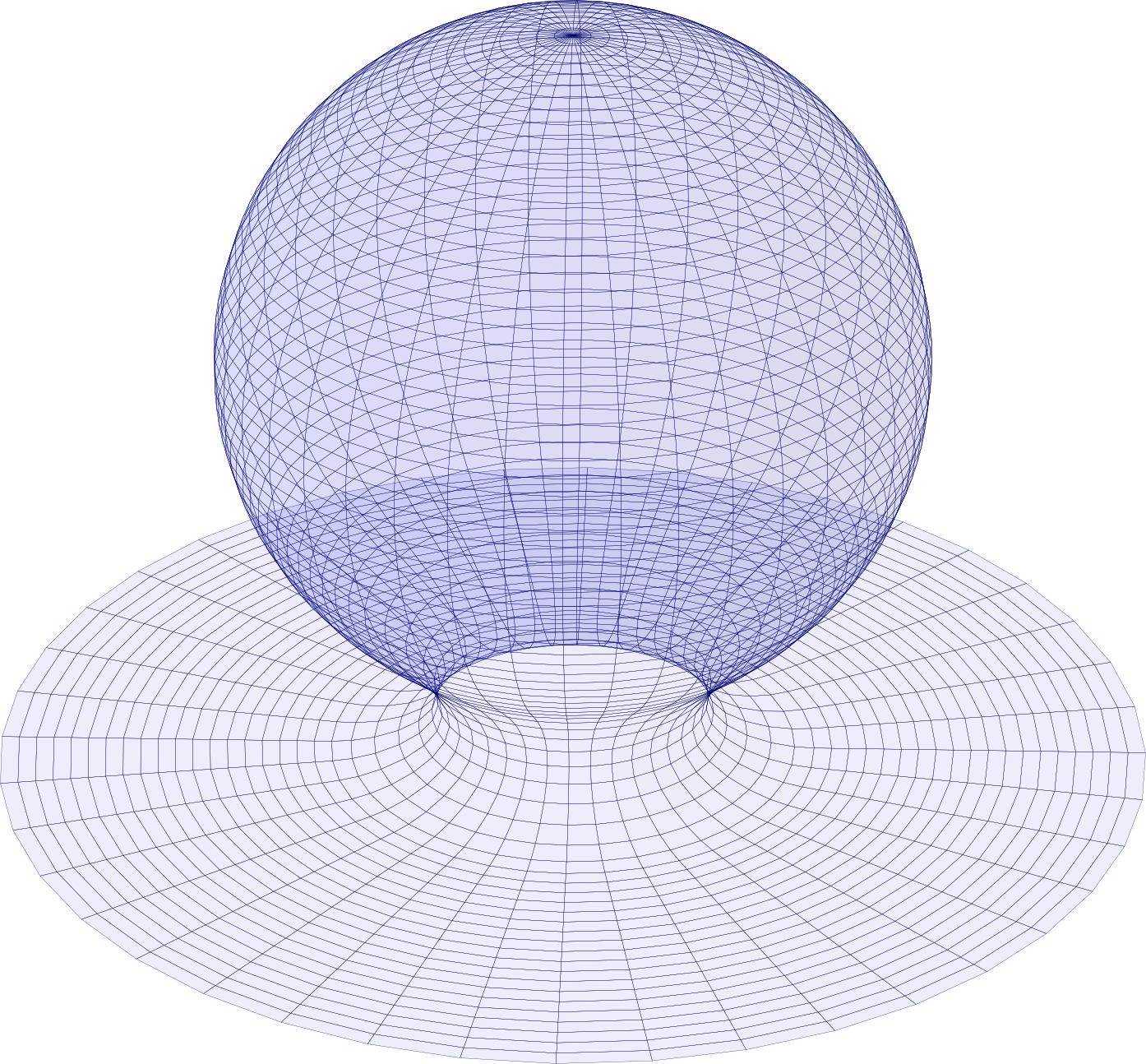}
  \caption{\it Example of a manifold which is diffeomorphic to $\mathbb{R}^m$, obtained by deforming a sphere $\mathbb{S}^m$ (visualized for $m=2$). For dimension $m \ge 5$ we give an example of a random variable on such a manifold which features a smeary mean. \label{fig:deformed_sphere}}
\end{figure}

We show 
\begin{itemize}
  \item that \emph{cut locus smeariness} occurs on the circle and the torus and can therefore be understood exhaustively by studying the properties of the random variable in a neighborhood of the cut locus.
  \item that \emph{geometrical smeariness} occurs on every m-dimensional sphere
  \begin{align*}
  \mathbb{S}^m := \left\{ x \in \mathbb{R}^{m+1} \, : \, \| x \| = 1 \right\}
  \end{align*}
  with $m \ge 2$, which is a novel phenomenon.
  \item in Theorem \ref{theorem:no-crit} that in contrast to cut locus smeariness, a unique mean with standard $n^{-1/2}$ asymptotic rate is compatible with arbitrarily high probability density at the cut locus.
  \item in Theorem \ref{theorem:geo-smeary} explicit examples of random variables for $m \ge 5$, whose ranges feature a finite sized spherical hole around the antipodal point of the mean, which exhibit a smeary asymptotic rate of $n^{-1/6}$.
  \item in a corollary of Theorem \ref{theorem:geo-smeary} that the sphere can be deformed on the hemisphere opposite of the mean to eliminate the cut locus of the mean altogether, for example by cutting out a ball around the cut locus and gluing the resulting boundary to a Euclidean space with a smooth transition, as illustrated in Figure \ref{fig:deformed_sphere}. The resulting manifold is then diffeomorphic to Euclidean space.
  \item in Theorem \ref{theorem:curse} a curse of dimensionality due to the increase with dimension $m$ of the maximal hole radius $\frac{\pi}{2} - \frac{K}{m}$ which still allows for an $n^{-1/6}$ rate, where $K$ is a constant: in the limit of infinite dimensional spheres, random variables with range barely larger than a hemisphere feature smeariness.
\end{itemize}

In Section \ref{sec:basics}, we give precise definitions of cut locus and geometrical smeariness. In Section \ref{sec:geo_smeary} we give an example of random variables on spheres of dimension $m \ge 5$ whose range has a hole containing the cut locus of the mean, and whose mean displays geometrical smeariness. In Sections \ref{sec:basics} and \ref{sec:geo_smeary} we state conjectures beyond this paper. In Section \ref{sec:application} we investigate simulated landmark shapes to illustrate some implications of smeariness in that setting. Furthermore, we analyze 151 real data sets of geomagnetic pole orientations on $\mathbb{S}^2$ and find smeariness in 17 of the data sets.

\section{Cut Locus Smeariness and Geometrical Smeariness}\label{sec:basics}

We start with geometrical context and previous results on smeariness and move on to new definitions.

\subsection{Basic Notions}

First, we introduce some basic notions, which will be used throughout the text. Let $\Omega$ be a probability space and let $Q$ be a Riemannian manifold called the \textit{data space} with the corresponding geodesic distance $d_Q(q, q') : Q \times Q \to \mathbb{R}$. Let $X : \Omega \to Q$ be a $Q$-valued random variable and $X_1, \dots, X_n \iid X$. In order to formulate a CLT in terms of random vectors, we map to the tangent space $T_{q_0} Q$ of some point $q_0 \in Q$ using the exponential map.

\begin{Def} \label{def:cut-locus}
  Consider a point in a Riemannian manifold $p \in Q$. The \emph{cut locus} $\cut(p)$ of $p$ is the closure of the set of all points $q \in Q$ such that there is more than one shortest geodesic from $p$ to $q$.
\end{Def}

\begin{Def}
  For a Riemannian manifold $Q$ and a point $q \in Q$ we define the exponential map $\exp_q : T_{q} Q \to Q$. This is the unique map with $\exp_{q}(0)= q$ and for any $v \in T_{q} Q$, considering the arc length parametrized geodesic $\gamma$ with $\gamma(0) = q$ and $\gamma'(0) = \frac{v}{|v|}$ we have $\exp_{q}(v) = \gamma(|v|)$. The inverse of the exponential map $\exp_q$, which exists outside of $\cut(q)$, is called the \emph{logarithm map} and is denoted by $\log_q$. It maps the point $p \in Q \setminus \cut(q)$ to a vector in the tangent space $T_q Q$ whose length is the same as the geodesic distance between $q$ and $p$.
\end{Def}

For the circle and spheres of arbitrary dimension, the cut locus of a point $p$ is simply its antipodal point. In general, the cut locus of a manifold of dimension $m$ has dimension at most $m-1$.

\begin{Def}[Population Fr\'echet mean]\label{def:frechet_mean}
  The set of population \emph{Fr\'echet means} of the random variable $X$ in $Q$ is defined as
  \begin{align*}
  E &= \left\{ \mu \in Q : \mathbb{E}[d^2_Q(\mu, X)] = \inf_{q \in Q} \limits \mathbb{E}[d^2_Q(q, X)] \right\} \, .
  \end{align*}
\end{Def}

\begin{Def}[Local and Global Fr\'echet mean]
  Local minima of the function $q \mapsto \mathbb{E}[d^2_Q(q, X)]$ are called \emph{local Fr\'echet means}. For clearer distinction, the Fr\'echet means as defined in Definition \ref{def:frechet_mean} will sometimes be called \emph{global Fr\'echet means} in the following.
\end{Def}

For readers from different fields it seems in order to give some historical context of the terminology used here. \cite{KWS90} has introduced the term \emph{Karcher mean} for local minima of the Fr\'echet function $q \mapsto \mathbb{E}[d^2_Q(q, X)]$, a fact criticized by the inadvertent name patron in \cite{Ka14}, who points out that the term \emph{Riemannian Center of Mass} is common in differential geometry. \cite{Ka14} also criticizes the use of the term \emph{Fr\'echet mean} for global minima, but since it is established and widely used in statistical literature, the present article will continue using it and use the term \emph{local Fr\'echet mean} for local minima.

\begin{As}
  In all of the following, we assume the random variables $X$ to have
  \begin{enumerate}[(i)]
    \item a unique population Fr\'echet mean $E = \{\mu\}$.
    \item a density in a neighborhood of $\cut(\mu)$.
  \end{enumerate}
\end{As}
Providing conditions for uniqueness of the Fr\'echet mean is a difficult and ongoing issue not further discussed here, cf. \cite{Ka77,KWS90,L01,Gr05,Afsari10,AM14,HH15}.

\begin{Def}[Fr\'echet function in exponential chart] \label{def:local-manifold}
  Consider a neighborhood $\widetilde{U}$ of $\mu$, $m\in \mathbb{N}$, such that with a neighborhood $P \subset T_{\mu} Q$ of the origin in $T_{\mu} Q \cong \mathbb{R}^m$ the exponential map $\exp_{\mu} : P \to \widetilde{U}$, $\exp_{\mu}(0)= \mu$, is a diffeomorphism. We set for $q\in Q$, $x \in P$,
  \begin{align*}
  \rho &: (x,q) \mapsto d^2_Q(\exp_{\mu}(x), q)\, , &  F &: x \mapsto \mathbb{E}[\rho(x, X)]\, .
  \end{align*}
  The function $F$ is called population \emph{Fr\'echet function} in the exponential chart at $\mu$. Since $\log_\mu(\mu) = 0$, $F$ has a global minimum at $x=0$.
\end{Def}

\begin{Notation}
  For a point $q \in Q$ and $\eps > 0$ let $B_\eps(q) = \{q' \in Q  : d_Q(q, q') < \eps \}$.
\end{Notation}

%\begin{Rm}[Local Fr\'echet mean]
%  \cite{KWS90} has introduced the term \emph{Karcher mean} for local minima of the Fr\'echet function, a fact criticized by the inadvertent name patron in \cite{Ka14}, who points out that the term \emph{Riemannian Center of Mass} is common in differential geometry. \cite{Ka14} also criticizes the use of the term \emph{Fr\'echet mean} for global minima, but since it is established and widely used in statistical literature, the present article will continue using it and use the term \emph{local Fr\'echet mean} for local minima.
%\end{Rm}

\subsection{Smeariness}\label{sec:smeariness}

A definition of smeariness was given in \cite{EH19}. Here, we provide a definition which highlights the point that smeariness is dependent on the random variable. Different asymptotic rates for different random variables due to different leading orders of the Fr\'echet function at the population mean can be understood in the context underlying \cite{vdV00}~Theorem~5.52.

\begin{Def}[Smeariness of Random Variables] \label{def:smeary}
  Consider a random variable $X$ on $Q$ with Fr\'echet function $F$ and Fr\'echet mean $\mu$. Assume that there is $\zeta > 0$ such that for every $x \in B_\zeta (0) \setminus \{0\}$ one has $F(x) > F(0)$. Suppose that for fixed constants $C_X>0$ and $2 < \kappa \in \mathbb{R}$ and a linear subspace $\mathcal{V} \subseteq T_\mu Q$ we have for every sufficiently small $\delta > 0$
  \begin{align*}
  \sup_{x \in \mathcal{V}, \, \|x\| < \delta} \limits \left| F(x) - F(0) \right| &\ge C_X \delta^\kappa \, .
  \end{align*}
  Then we say that the Fr\'echet mean of $X$ is \emph{smeary on the linear subspace $\mathcal{V}$} and that \emph{$Q$ admits smeariness}. If $\mathcal{V} = T_\mu Q$, we simply say that $X$ is \emph{smeary}.
\end{Def}

Note that we do not require the Fr\'echet function to be analytic. The smeary asymptotic theory relies on the stricter \cite[Assumption 2.6]{EH19} in order to get a closed form for the asymptotic distribution. However, it is applicable e.g. to the Fr\'echet function $F(x) = |x|^4 + \exp(-|x|^{-1})$, which is smooth but non-analytic.

For the definition of geometrical and cut locus smeariness, we introduce some auxiliary notation. Let $P = Q \setminus \cut(\mu)$ and $d_{P}$ the geodesic distance on $P$. Then we write
\begin{align*}
\tau &: (x,q) \mapsto d^2_{P}(\exp_{\mu}(x), q)\, , & G &: x \mapsto \mathbb{E}[\tau(x, X)] \, .
\end{align*}
Here, $d^2_{P}(q,p) \ge d^2_{Q}(q,p)$ is given by the infimum over the length of all curves in $P$ connecting $q$ and $p$. Using this, we can define

\begin{Def}[Cut Locus Smeariness and Geometrical Smeariness]\label{def:top-geo}
  Assume that $Q$ admits smeariness and $X$ is a random variable with smeary Fr\'echet mean $\mu \in Q$ on the linear subspace $\mathcal{V} \subseteq T_\mu Q$. The restriction of the Hesse matrix to $\mathcal{V}$ is denoted by $\Hess_{\mathcal{V}}$.
  \begin{enumerate}[(i)]
    \item If for every linear subspace neighborhood $U \subset \mathcal{V}$ of $0$, there is an $x \in U$ such that $F(x) \neq G(x)$ and $\Hess_{\mathcal{V}}(F-G)(0) < 0$, then the mean of $X$ is called \emph{cut locus smeary on the linear subspace $\mathcal{V}$} and we say that \emph{$Q$ admits cut locus smeariness}.
    \item If there is a linear subspace neighborhood $U \subset \mathcal{V}$ of $0$, such that for every $x \in U$ one has $F(x) = G(x)$ or if $\Hess_{\mathcal{V}}(F-G)(0) \ge 0$, the mean of $X$ is called \emph{geometrically smeary on the linear subspace $\mathcal{V}$} and we say that \emph{$Q$ admits geometrical smeariness}.
  \end{enumerate}
\end{Def}

\noindent To motivate the term \emph{cut locus smeariness}, first note that smeariness always implies that $\Hess_{\mathcal{V}} F(0)$. On Euclidean space, the Hesse matrix of the Fr\'echet function is always positive definite. In reference to this, smeariness depends on a negative contribution to the Hessian that leads to a vanishing Hessian overall. The terms \emph{cut locus smeariness} and \emph{geometrical smeariness} point to the origin of this negative contribution to the Hessian. Note that one can write
\begin{align*}
0 = \Hess_{\mathcal{V}} F(0) = \underbrace{\Hess_{\mathcal{V}} G (0)}_{> 0} + \underbrace{\Hess_{\mathcal{V}}(F-G)(0)}_{< 0}
\end{align*}
to illustrate that cut locus smeariness crucially hinges on $F \neq G$. Since the only difference between these two functions is that in $G$ the geodesics crossing $\cut(\mu)$ are excluded, the negative term $\Hess(F-G)(0)$ can be understood as the contribution of the cut locus.

The author is not aware of any random variables on any space, where $\Hess(F-G)(0) > 0$. If such exist and their means exhibits smeariness, we would classify their means as geometrically smeary, since the defining property of random variables with a smeary mean is a negative contribution to the Hessian. However, in such a case, the cut locus contribution to the Hessian would be positive thus the negative contribution to the Hessian which causes smeariness is from a different source.

Definition \ref{def:top-geo} does not exclude a space admitting both geometrical and cut locus smeariness both on orthogonal subspaces as well as on a common linear subspace. However, Theorem \ref{thm:circle_sphere_smeary} shows that the circle and the spheres $\mathbb{S}^m$ with $m \ge 2$ do not admit both. There is no known example of a space admitting geometrical and cut locus smeariness on a common linear subspace. Geometrical and cut locus smeariness on orthogonal subspaces can be realized on product spaces like $\mathbb{S}^1 \times \mathbb{S}^5$.

\begin{Theorem}\label{thm:circle_sphere_smeary}$ $
  \begin{enumerate}[(i)]
    \item The circle only admits cut locus smeariness.
    \item The spheres $\mathbb{S}^m$ with $m \ge 2$ only admit geometrical smeariness.
  \end{enumerate}
\end{Theorem}

\begin{proof}$ $
  \begin{enumerate}[(i)]
    \item On the circle, the cut locus of the mean contains only one point, namely its antipode, which we denote by $\oli{\mu}$. In \cite{HH15} it was shown that $\Hess G (0) = 2$ and $\Hess(F-G)(0) = -4\pi f(\oli{\mu})$ where $f$ is the probability density that exists in the neighborhood of $\oli{\mu}$. Thus, only cut locus smeariness exists, namely if $f(\oli{\mu}) = \frac{1}{2\pi}$.
    \item The set of geodesic segments connecting two points of $P$ consists of all geodesic segments of $Q$ which do not cross $\cut(\mu)$. Thus, $d_{P}(q, p) = d_{Q}(q, p)$ unless the shortest geodesic segment connecting $q$ and $p$ crosses $\oli{\mu}$. In that case, there is no shortest curve connecting $q$ and $p$, but curves can get arbitrarily close in length to the geodesic segment in $Q$, thus also for these points $d_{P}(q, p) = d_{Q}(q, p)$. In consequence, there is a neighborhood $U$ of $0$, such that for every $x \in U$ one has $F(x) = G(x)$ and thus only geometrical smeariness occurs. \qedhere
  \end{enumerate}
\end{proof}

Note that the argument in part (ii) of the above proof can be generalized to show that $d_{P}(q, p) = d_{Q}(q, p)$ for all $p,q \in Q$ whenever the dimension of the cut locus of the mean is smaller than $m-1$, since then the ``infinitesimal circumvention argument'' employed in the proof is applicable. In consequence, cut locus smeariness is only possible if the cut locus has dimension $m-1$.

Conversely, one may expect cut locus smeariness to be possible if the cut locus dimension is $m-1$, observing the following heuristic argument. Let $\gamma$ be the point set of a geodesic segment from $\mu$ to an internal point of $\cut(\mu)$, then for all points $q \in B_\eps(\mu) \cap \gamma \setminus \{\mu\}$ there are points $p$ such that $d_{P}(q, p) > d_{Q}(q, p)$ and in consequence $\Hess_{\textnormal{span}\{\dot{\gamma}(\mu)\}}(F-G)(0) < 0$ for random variables $X$ with a non-vanishing density in a neighborhood of $\cut(\mu)$, using a variant of the argument used on the circle. This leads to the following conjecture.

\begin{Conj}
  In the real projective spaces $\mathbb{R}P^m$ with $m \ge 2$, for every neighborhood $U$ of $0$ there is an $x \in U$ such that $F(x) \neq G(x)$. This indicates that the $\mathbb{R}P^m$ admit cut locus smeariness. 
\end{Conj}

For smeariness on the circle, a probability density at $\oli{\mu}$ of $f(\oli{\mu}) = \frac{1}{2\pi}$ is necessary for smeariness. Theorem \ref{thm:circle_sphere_smeary} shows that this is closely tied to the fact that smeariness on the circle is cut locus smeariness. One may therefore suspect that on $\mathbb{S}^m$ with $m \ge 2$ there is no such value for the probability density at $\oli{\mu}$ which leads to smeariness. In fact, we show the following:

\begin{Theorem} \label{theorem:no-crit}
  For every real $\rho \ge 0$ and every integer $m \ge 2$ one can define a random variable $X_\rho$ on $\mathbb{S}^m$ which has a probability density with value $\rho$ at the south pole and a non-smeary mean at the north pole.
\end{Theorem}

\begin{proof}
  See supplement A.2.1.
\end{proof}

\section{Geometrical Smeariness on Spheres}\label{sec:geo_smeary}

In Theorem \ref{theorem:no-crit} it was shown that the value of the probability density at $\cut(\mu)$ is not sufficient to achieve smeariness on $\mathbb{S}^m$ for $m \ge 2$, in contrast to the situation on $\mathbb{S}^1$. In this Section, we address the converse question, whether the probability density which the random variables presented by \cite{EH19} exhibit at $\cut(\mu)$ are necessary for smeariness. Investigating this question aims to shed light on the assumption that the Hessian of the Fr\'echet function be non-singular, which is still a staple of efforts such as \cite{BL17,EGHT19} to generalize prerequisites for the CLT. Let $\mu := e_{m+1}$ be the north pole of $\mathbb{S}^m$ and write the spherical annulus as
\begin{align*}
\mathbb{L}_{m,\beta} := \{q\in \mathbb{S}^m: \arccos\left< p, \mu \right> \in [\pi/2, \pi - \beta]\} \, .
\end{align*}
This set is the southern hemisphere with a \emph{hole of radius $\beta$} around the south pole cut out.

\begin{Theorem} \label{theorem:geo-smeary}
  Consider a random variable $X$ on the $m$-dimensional unit sphere $\mathbb{S}^m$ ($m\ge 5$) that is uniformly distributed on $\mathbb{L}_{m,\beta}$ with total mass $0<\alpha<1$ and assuming $\mu$ with probability $1-\alpha$. Then there is a radius $\beta_0 > 0$ of a spherical hole such that the random variable has a unique 2-smeary Fr\'echet mean, i.e. with asymptotic rate $n^{-1/6}$, at the north pole for any $\beta \le \beta_0$.
\end{Theorem}

\begin{proof}
  Due to rotation symmetry, the Fr\'echet function only depends on the angle $\psi$ between $\mu$ and $\exp_\mu(x)$. Including the parameters $\alpha$ and $\beta$ we write $F(\alpha, \beta, \psi)$. For given $\beta \le \frac{\pi}{2}$ we define $\alpha_\beta$ as the value of $\alpha$ such that $\frac{\partial^2F}{\partial\psi^2} (\alpha_\beta, \beta, 0) = 0$.
  
  The proof proceed in two steps. First, we show that there is a $\beta_{m,4} > 0$ such that the Hessian of the Fr\'echet function vanishes at the north pole but the fourth derivative is positive, such that we have a local Fr\'echet mean. This is established by Lemmas A.3 and A.5 in the supplement and further elaborated in Theorem \ref{theorem:curse} below.
  
  Secondly, we show that there is a $\beta_{0} > 0$ such that for all $\beta \le \beta_0$ the local Fr\'echet mean at the north pole is the unique Fr\'echet mean. For this we will utilize Lemmas A.1 and A.8 from the supplement. From \cite{EH19} we recall $\frac{\partial^4F}{\partial\psi^4} (\alpha_0, 0, 0) =\frac{\alpha_0 v_{m+1}}{v_m}\,\frac{m-1}{m+2}=c_m>0$. Furthermore, from Lemma A.1, one gets $\frac{\partial^2F}{\partial\psi^2} (\alpha_0, 0, \psi) > 0$ for $\psi \neq 0,\pi$. Hence we infer that $\frac{\partial F}{\partial\psi} (\alpha_0, 0, \psi)$ is strictly increasing in $\psi$ from $\frac{\partial F}{\partial\psi} (\alpha_0, 0, 0) =0$, yielding that there is no stationary point for $F$ other than $p=\mu$.
  
  Due to Lemmas A.1 and A.8 we know for all $\psi \le \pi/3$
  \begin{align*}
  \frac{\partial^4F}{\partial\psi^4} (\alpha_\beta, \beta, \psi) \ge \frac{\partial^4F}{\partial\psi^4} (\alpha_0, 0, \psi) - L_4 \beta \ge \frac{c_m}{2} - L_4 \beta \, .
  \end{align*}
  Thus we can pick $\beta \le \frac{c_m}{2 L_4}$ to get $\frac{\partial^4F}{\partial\psi^4} (\alpha_\beta, \beta, \psi) \ge 0$ for all $\psi \le \pi/3$. Since $\frac{\partial^3F}{\partial\psi^3} (\alpha_\beta, \beta, 0) = 0$ and $\frac{\partial^2F}{\partial\psi^2} (\alpha_\beta, \beta, 0) = 0$ it follows that $\frac{\partial^3F}{\partial\psi^3} (\alpha_\beta, \beta, \psi) > 0$ and $\frac{\partial^2F}{\partial\psi^2} (\alpha_\beta, \beta, \psi) > 0$ for all $0 < \psi \le \pi/3$.
  
  From Lemma A.8 we note that
  \begin{align*}
  \left| \frac{\partial^2F}{\partial\psi^2} (\alpha_0, 0, \psi) - \frac{\partial^2F}{\partial\psi^2} (\alpha_\beta, \beta, \psi) \right| &\le L_2 |\beta | \, .
  \end{align*}
  Thus we can pick $\beta < \frac{1}{L_2} \frac{\partial^2F}{\partial\psi^2} (\alpha_0, 0, \pi/3)$ to achieve $\frac{\partial^2F}{\partial\psi^2} (\alpha_\beta, \beta, \pi/3) \ge \frac{\partial^2F}{\partial\psi^2} (\alpha_0, 0, \pi/3) - L_2 \beta > 0$. Since $\frac{\partial^2F}{\partial\psi^2} (\alpha_0, 0, \psi)$ is monotonously growing in $\psi$, it follows that $\frac{\partial^2F}{\partial\psi^2} (\alpha_\beta, \beta, \psi) >0 $ for all $\psi \ge \pi/3$. Thus for all
  \begin{align}
  \beta < \beta_0 = \min \left( \frac{c_m}{2 L_4}, \frac{1}{L_2} \frac{\partial^2F}{\partial\psi^2} (\alpha_0, 0, \pi/3) \right)
  \end{align}
  the minimum at $\psi = 0$ is unique.
\end{proof}

In the case of geometrical smeariness, it is possible to have smeary means even for random variables whose range excludes a neighborhood of the cut locus of the mean. This illustrates an especially drastic consequence of geometrical smeariness which distinguishes it from cut locus smeariness.

\begin{Cor} \label{cor:deformed}
  There are manifolds which are diffeomorphic to $\mathbb{R}^m$ ($m\ge 5$) on which the mean can be smeary.
\end{Cor}

\begin{proof}
  Consider the random variable on $\mathbb{S}^m$ used in Theorem \ref{theorem:geo-smeary}, cut out a ball of radius $r < \beta_0$ around the south pole and glue the resulting boundary to a Euclidean space with a smooth connection. The resulting manifold is then diffeomorphic to Euclidean space.
\end{proof}

This result points out particularly clearly that smeariness can be completely independent of the behavior of the random variable at the cut locus of the mean because smeariness can even occur in cases where the mean does not have a cut locus and the manifold is topologically trivial.

For the statement of the following Theorem, note that Definitions \ref{def:smeary} and \ref{def:top-geo} can be extended to smeary local Fr\'echet means.% Here, we show a curse of dimensionality for smeary local Fr\'echet means on $\mathbb{S}^m$.

\begin{Theorem}[Curse of dimensionality] \label{theorem:curse}
  For every radius $\frac{\pi}{2} < r < \pi$ there is a dimension $m$ such that one can construct a random variable $X$ on $\mathbb{S}^m$ whose range has radius $r$ and which has a smeary local Fr\'echet mean. For $m \to \infty$ the minimal radius of the range of $X$ sufficient for smeariness of local Fr\'echet means approaches $\frac{\pi}{2}$ with a rate of $m^{-1}$. For arbitrarily small $\epsilon > 0$,
  \begin{enumerate}[(i)]
    \item it is possible to have random variables on $\mathbb{S}^m$ whose range is restricted to $\theta \in \left[0, \frac{\pi}{2} + \frac{16 + \epsilon}{\pi(m-3)}\right]$ and which exhibit a smeary local Fr\'echet mean at the north pole.
    \item for the random variables used in Theorem \ref{theorem:geo-smeary} for hole sizes $\beta \le \frac{\pi}{2} - \frac{6(6+\pi) + \epsilon}{\pi(m-3)}$ there are parameter values $\alpha = \alpha_\beta$ which lead to local smeary means at the north pole.
  \end{enumerate}
\end{Theorem}

\begin{proof}
  \begin{enumerate}[(i)]
    \item From supplement A.3 we see that $\frac{\partial^2 F_\theta}{\partial\psi^2}|_{\psi=0} < 0$ if and only if $\theta > \theta_{m,2}$ and $\frac{\partial^4 F_\theta}{\partial\psi^4}|_{\psi=0} > 0$ if and only if $\theta > \theta_{m,4}$. From Lemma A.4 we can thus derive the sufficient condition
    \begin{align*}
    \theta > \frac{\pi}{2} + \frac{16}{\pi(m-3)} \ge \theta_{m,4} \ge \theta_{m,2} \, .
    \end{align*}
    Adding a point mass at the north pole, one can thus achieve $\frac{\partial^2 F}{\partial\psi^2}|_{\psi=0} = 0$ and $\frac{\partial^4 F}{\partial\psi^4}|_{\psi=0} > 0$, which yields a smeary local Fr\'echet mean. The claim follows.
    \item From Lemma A.7 we can derive the sufficient condition
    \begin{align*}
    \beta < \frac{\pi}{2} - \frac{6(6+\pi)}{\pi(m-3)} \le \beta_{m,4} \le \beta_{m,2}
    \end{align*}
    for $\frac{\partial^2 F}{\partial\psi^2}|_{\psi=0} = 0$ and $\frac{\partial^4 F}{\partial\psi^4}|_{\psi=0} > 0$, which yields a smeary local Fr\'echet mean. The claim follows.
  \end{enumerate}
\end{proof}

%\begin{Rm} \label{rmk:curse}
%  Theorem \ref{theorem:curse} implies that for every radius $\frac{\pi}{2} < r < \pi$ (for example $r = \pi - \beta$ in Theorem \ref{theorem:curse} (ii)) there is a dimension $m$ such that one can construct a random variable $X$ on $\mathbb{S}^m$ whose range has radius $r$ and which has a smeary local Fr\'echet mean. More precisely, for $m \to \infty$ the minimal radius of the range of $X$ sufficient for smeariness of local Fr\'echet means approaches $\frac{\pi}{2}$ with a rate of $m^{-1}$.
%\end{Rm}

\begin{Conj}
  Numerical calculations show that the local Fr\'echet means of Theorem \ref{theorem:curse} are actually global Fr\'echet means. This indicates that Theorem \ref{theorem:curse} holds also for global Fr\'echet means.
\end{Conj}

Summarizing the results, we find that there is neither a necessary nor sufficient probability density of a random variable $X$ at $\cut(\mu)$ for smeariness of $\mu$ on $\mathbb{S}^5$. In fact, for very large dimension, smeariness of the mean of $X$ is still possible when the support of $X$ only minimally exceeds a hemisphere.

\section{Consequences for Applications} \label{sec:application}

Smeariness is a serious problem for statistics on spheres because if a data set is sampled from a population with smeary mean, this greatly reduces the power of hypothesis tests and makes the sample Fr\'echet mean less reliable. Smeariness of the underlying population leads to \emph{finite sample smeariness} of samples, as elaborated in \cite{HEH19}. We provide a definition for finite sample smeariness and show that it can occur under fairly general conditions and consider two applications. The first is a set of simulated oriented landmark shapes as an example of a smeary distribution. The second application concerns real data, namely north pole positions on the earth during periods of pole reversal which have been determined from magnetite rock samples.

In the following we use sample Fr\'echet means, which are defined as follows.

\begin{Def}[Sample Fr\'echet mean]\label{def:frechet_mean_sample}
  The set of sample \emph{Fr\'echet means} of the random variable $X$ in $Q$ is defined as
  \begin{align*}
  E_n(\omega) &= \left\{ \widehat{\mu}_n \in Q : \sum_{i=1}^{n} \limits d^2_Q(\widehat{\mu}_n, X_i(\omega)) = \inf_{q \in Q} \limits \sum_{i=1}^{n} \limits d^2_Q(q, X_i(\omega)) \right\} \, ,
  \end{align*}
  where the argument $\omega$ will be suppressed in the following.
\end{Def}

\subsection{Finite Sample Smeariness}

As illustrated in \cite{EH19}, random variables which satisfy a CLT with standard $n^{-1/2}$ asymptotic rate can still exhibit higher than expected variance of sample means for finite sample sizes. This phenomenon is called \emph{finite sample smeariness} and is discussed in depth on the circle and the torus in \cite{HEH19}. For any random variable $Y$, denote $\textnormal{Var}[Y] := \mathbb{E}\left[d(Y,\mu)^2\right]$.

\begin{Def}\label{def:FSS}
  Given constants $C_+, C_-, K > 0$, $0 < r_- < r_+ < 1 $ and integers $1 < n_- < n_+ < n_0$ satisfying $C_+ n_{-}^{r_+}\leq C_- n_{+}^{r_-}$, the Fr\'echet mean $\mu$ of a random variable $X$ is called \emph{finite sample smeary} if the following holds for the Fr\'echet sample mean $\widehat{\mu}_n$
  \begin{itemize}
    \item[(i)] $\forall n \in (n_-, n_+] \cap \mathbb{N} \, : \quad \textnormal{Var}[X]\leq C_- n^{r_-} \le n \textnormal{Var}[\widehat{\mu}_n] \le C_+ n^{r_+}$.
    \item[(ii)] $\forall n \in (n_0, \infty) \cap \mathbb{N} \, : \quad \textnormal{Var}[\widehat{\mu}_n] \le K n^{-1}$.
  \end{itemize}
\end{Def}

Since finite sample smeariness is a non-asymptotic property, which is strictly weaker than smeariness, alternative terminology like \emph{lethargic means} has been proposed for clearer distinction. Here, we decide to stick with the terminology by \cite{HEH19}. On the circle, finite sample smeariness occurs whenever the range of the random variable exceeds a half circle. For positively curved spaces, already the results by \cite{BP05,Afsari09} suggest that finite sample smeariness is a generic feature. Recent results by \cite{Pennec19} on small samples from highly concentrated measures show that in positively curved spaces
\begin{align*}
n \textnormal{Var}[\widehat{\mu}_n] > \textnormal{Var}[X]
\end{align*}
holds for all random variables except point masses. However, for concentrated random variables, the \emph{magnitude} of finite sample smeariness, which we define as
\begin{align*}
S_\textnormal{fs} = \max_{n \in \mathbb{N}} \limits \frac{n \textnormal{Var}[\widehat{\mu}_n]}{\textnormal{Var}[X]} > 1\, ,
\end{align*}
is mostly close to $1$. Here, we give a sufficient condition for $S_\text{fs}$ to be arbitrarily high.

\begin{Theorem}
  %Recall that $\theta_{m,4}$ is defined such that for uniform distributions on $\mathbb{S}^{m-1}$ at all $\theta > \theta_{m,4}$ the fourth derivative of the Fr\'echet function at the north pole is positive and the second derivative is negative.
  For $\mathbb{S}^m$ with $m \ge 4$, and for any $K > 1$ and any $\theta_* > \theta_{m,4}$ there is a random variable $X$, whose range is restricted to $\theta \in [0,\theta_*]$, which has a local Fr\'echet mean at the north pole and for which
  \begin{align*}
  S_\textnormal{fs} \ge \lim_{n \to \infty} \limits \frac{n \textnormal{Var}[\widehat{\mu}_n]}{\textnormal{Var}[X]} > K \, ,
  \end{align*}
  if the north pole is the unique global Fr\'echet mean.
\end{Theorem}

\begin{proof}
  See supplement A.5.
\end{proof}

In consequence of this theorem, one can expect far reaching consequences for data which are moderately spread out over the sphere. Due to the curse of dimensionality for $\theta_{m,4}$, arbitrarily high magnitude of finite sample smeariness is possible for random variable whose range is only slightly larger than a hemisphere in high dimension.

\subsection{Landmark Pre-Shapes}

In this example, we will look at contours described by 4 or 6 landmarks in $\mathbb{R}^2$. The contours are considered oriented with respect to some reference, as for example image elements in a computer image, which are oriented with respect to the frame, and the orientation is considered significant information. Therefore, we will factor out translations and scalings, thus mapping shapes with $k$ landmarks onto spheres $R^{2k} \to \mathbb{S}^{2k-3}$. Here, we consider $k = 4$ and $k = 6$ and thus $\mathbb{S}^5$ and $\mathbb{S}^9$. The shapes considered here are quadrangles, so 4 landmarks are sufficient to characterize them, as illustrated by the example shapes in Figure \ref{fig:shapes_examples}.

\begin{figure}[h!]
  \centering
  \subcaptionbox{Close to $\mu$, 4 landmarks}[0.45\textwidth]{\includegraphics[width=0.4\textwidth]{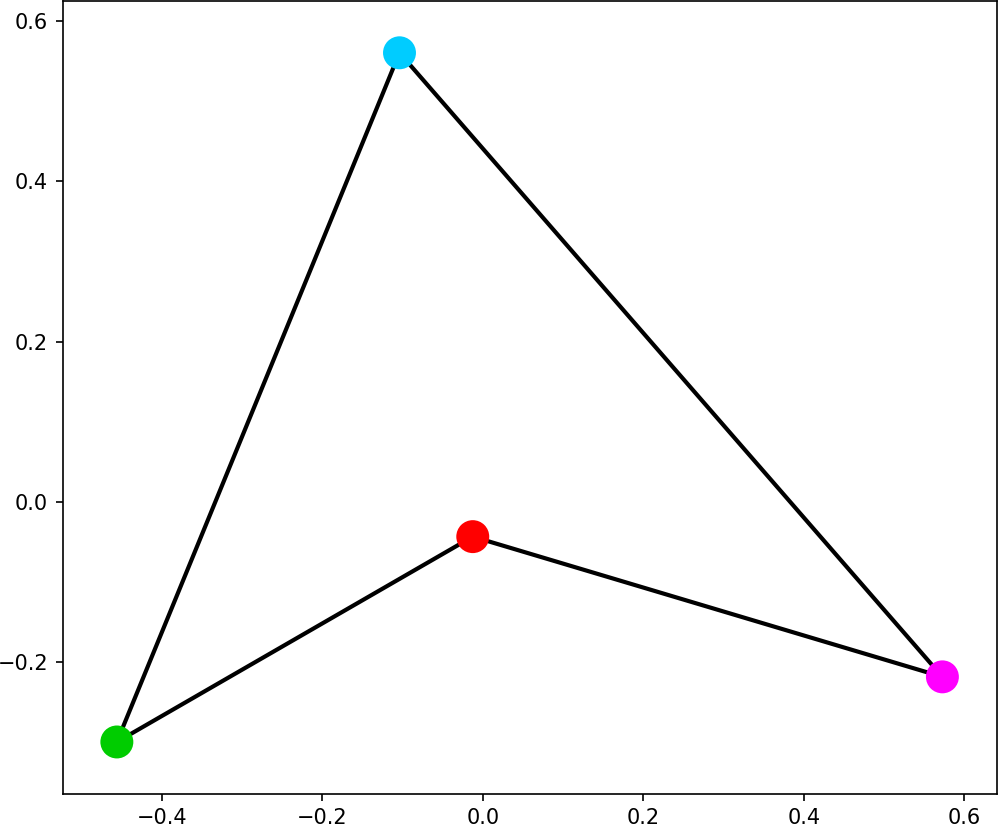}}
  \hspace*{0.02\textwidth}
  \subcaptionbox{Close to $\mu$, 6 landmarks}[0.45\textwidth]{\includegraphics[width=0.4\textwidth]{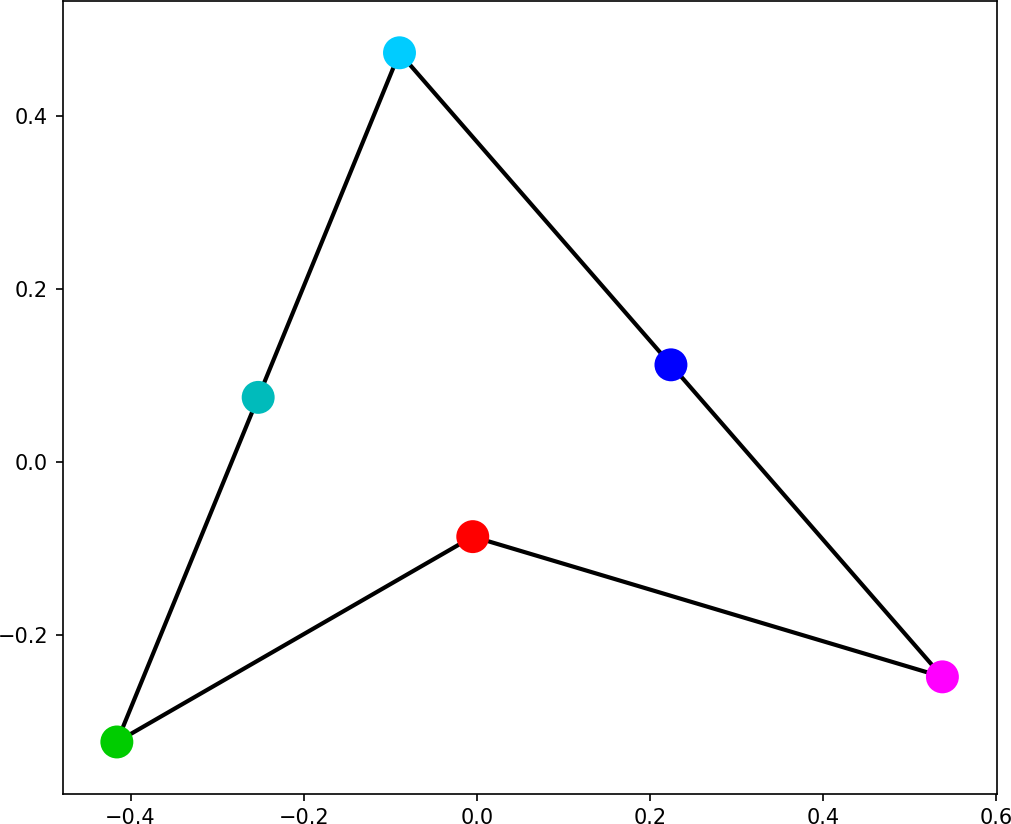}}\vspace*{\baselineskip}\\
  \subcaptionbox{Close to $\cut(\mu)$, 4 landmarks}[0.45\textwidth]{\includegraphics[width=0.4\textwidth]{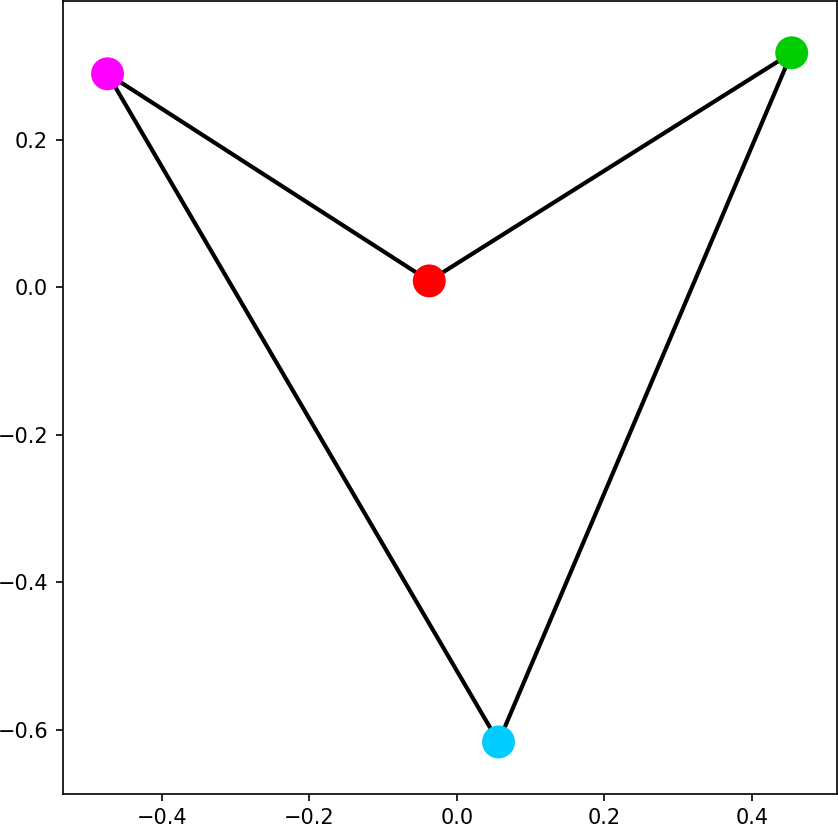}}
  \hspace*{0.02\textwidth}
  \subcaptionbox{Close to $\cut(\mu)$, 6 landmarks}[0.45\textwidth]{\includegraphics[width=0.4\textwidth]{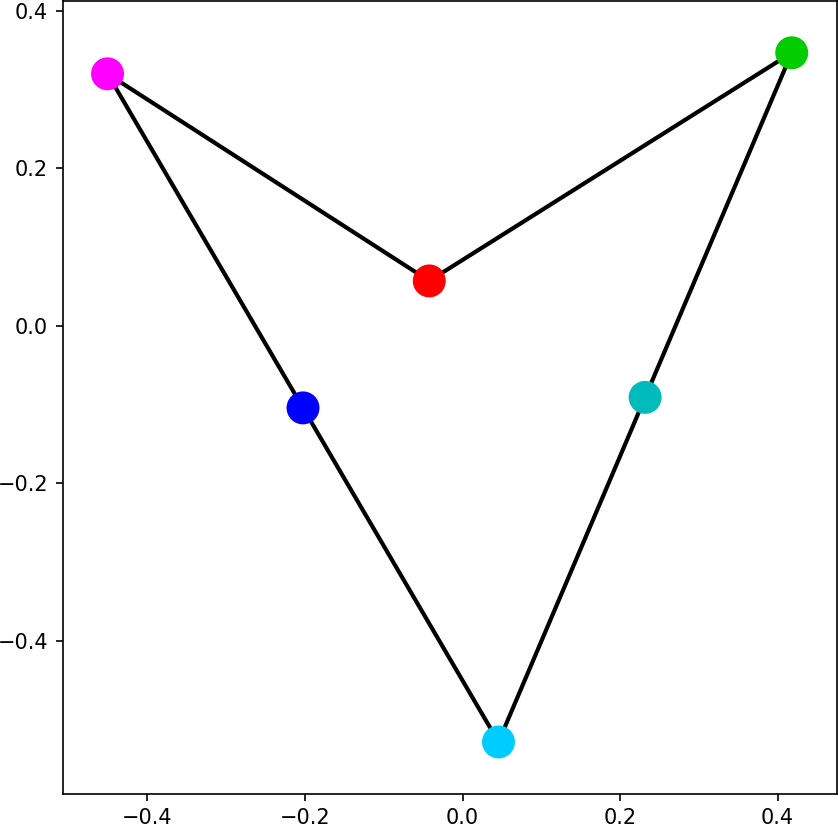}}
  \caption{Examples of the landmark pre-shapes used here. The shapes are quadrangles, therefore the additional landmarks do not highlight additional details of the shapes. However, if one includes the additional two landmarks, the landmarks are closer to being equidistant. \label{fig:shapes_examples}}
\end{figure}

The distribution on the 4 landmark $\mathbb{S}^5$ is rotation symmetric and given by
\begin{align*}
d \mathbb{P}(\theta) := (1 - \alpha) \mathbb{1}_{[0,0.05\pi]}(\theta) \, d\theta + \alpha \, d\delta_{0.95\pi}(\theta)\, .
\end{align*}
where the factor $\alpha$ is chosen such that the Hessian at the mean is slightly positive.

\begin{figure}[h!]
  \centering
  \includegraphics[width=0.7\textwidth]{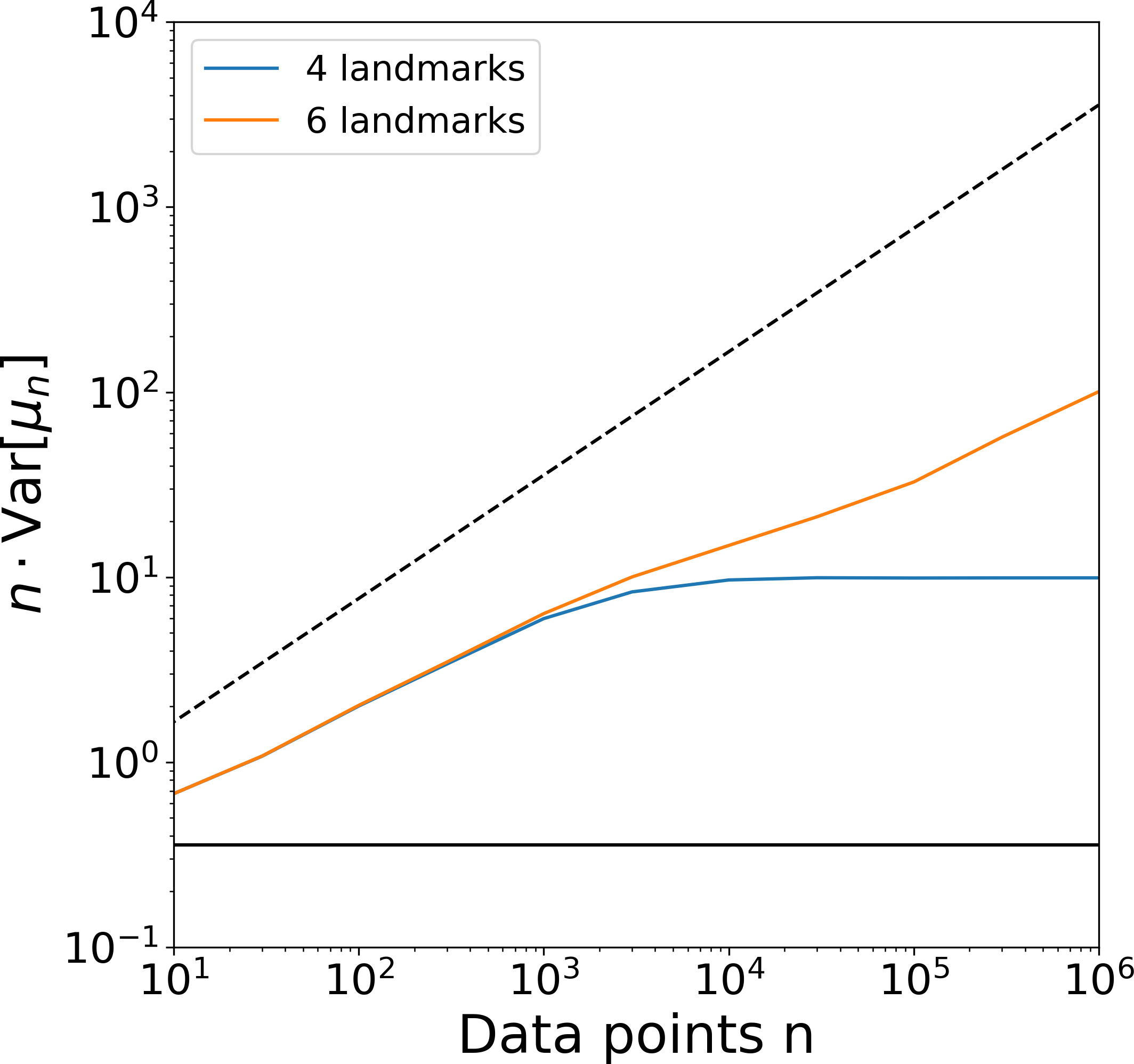}
  \caption{Variances of sample means for increasing sample size. For every sample size, $B=1000$ samples were drawn. One can clearly see that using 4 landmarks leads to standard asymptotics for large samples $n> 1000$, while the shapes characterized by 6 landmarks exhibit a lower rate of convergence to the population mean even for samples containing one million data points. Note that the vertical axis is rescaled by $n$ and we are considering variances, not standard deviations. Therefore, an asymptotic rate of $n^{-1/2}$ is represented by the horizontal solid black line, while a rate of $n^{-1/6}$ is represented by the dashed black line. In general, growing curves indicate higher asymptotic rates than $n^{-1/2}$. \label{fig:shapes_variances}}
\end{figure}

Looking at variances of means, it is apparent from Figure \ref{fig:shapes_variances} that the shapes with 6 landmarks exhibit a higher magnitude of finite sample smeariness and appear even compatible with smeariness. Looking at the 4-landmark shapes for increasing sample size, one can see that the variance the mean scales with close to $n^{-1/3}$, the asymptotic rate of 2-smeariness, for $n \le 1000$. For larger $n$, the scaling settles into the standard $n^{-1}$ rate. One might intuitively think that adding two additional landmarks, thus receiving the 6-landmark shapes, could improve the scaling behavior. However, Figure \ref{fig:shapes_variances} shows that the opposite is true in this case, as the variance scaling for 6-shapes remains slower than $n^{-1}$ for much larger sample sizes.\FloatBarrier

\subsection{Examples of Smeariness in Magnetic Pole Transitions}
% cite: LB13

The earth's magnetic field is dynamic and its north and south pole position on the surface move over time. As the magnetic field is generated by the earth's rotation, the poles are close to the rotation axis at most times. In the history of the earth, the positions of magnetic north and south poles have switched on the scale of few to several $100,000$ years, as evidenced by magnetization of certain minerals. Research in the field of paleomagnetism aims to discern magnetic polarization changes as well as past continental drift from the mineral record. We use a database of 151 data sets of reconstructed geomagnetic north pole positions (VGP, meaning ``Virtual Geomagnetic Pole'') around periods of polar transition presented in \cite{McEL96} and downloadable from \href{ftp://ftp.ngdc.noaa.gov/geomag/Paleomag/access/ver3.5}{ftp://ftp.ngdc.noaa.gov/geomag/Paleomag/access/ver3.5}.

\begin{table}[!ht]
  \caption{Age of the magnetized mineral, number of magnetic samples and source for all data sets which display smeariness. The upper part of the table lists data sets with strong finite sample smeariness, where $n \textnormal{Var}[\widehat{\mu}_n]$ increases up to large bootstrap sample size $k$ and the lower part lists data sets with less pronounced finite sample smeariness, mostly for small $k$.}
  \begin{center}
    \begin{tabular}{c|c|c|c}
      Number & Age [Ma] & n & Source\\
      \hline
      006   & 0.78 & 107 & \cite{CK87} \\
      012   & 0.78 &  40 & \cite{KNHK73} \\
      044   & 1.77 & 145 & \cite{CK85} \\
      062   & 1.77 & 154 & \cite{HK95} \\
      063   & 3.11 &  84 & \cite{L91} \\
      107   & 33.7 &  50 & \cite{WED69} \\
      122   & 1.07 & 239 & \cite{Rolph93} \\
      141   & 8.07 &  60 & \cite{Gu88} \\
      \hline\hline
      004   & 0.78 &  35 & \cite{CKO82} \\
      005   & 0.78 &  44 & \cite{CKO82} \\
      043   & 1.77 & 181 & \cite{CK87} \\
      061   & 3.04 & 193 & \cite{vHL92} \\
      077   & 4.8  & 177 & \cite{L88} \\
      109   & 180  &  16 & \cite{vZGH62} \\
      129   & 33.7 &  31 & \cite{WED69} \\
      133   & 4.18 & 193 & \cite{CSSR96} \\
      138   & 9.74 & 170 & \cite{Gu88} \\
    \end{tabular}
  \end{center}
  \label{tab:data_sets}
\end{table}

For each of these data sets we performed $k$-out-of-$n$ bootstrap for $1\le k \le 1000$ with $B=1000$ bootstrap replicates to get bootstrap sample means. From these we determined the variances of the means and looked at the power law behavior thereof. The expected behavior, if a standard CLT holds, would be $\mathrm{Var}[\mu^*_k] \propto k^{-1}$. Of the 151 data sets, 8 very clearly depart from this scaling behavior, exhibiting a clearly reduced variance of bootstrap sample means even for $k>n$. Another 9 data sets exhibit slower scaling at least for small $k$. A rough geological classification by age, sample sizes and sources for these 17 data sets are given in Table \ref{tab:data_sets}. Example data sets are displayed in Figure \ref{fig:smeary_magnet_examples} and bootstrap scaling behavior are shown in Figure \ref{fig:smeary_magnets_rates}.

For samples of finite size, one cannot determine with certainty, whether the underlying random variable has a finite density at the cut locus of the mean or not. As shown by \cite{LB13}, the cut locus of the mean cannot contain a data point, which means that a neighborhood of the cut locus is always free of probability mass. In consequence, smeariness with or without a hole at the cut locus cannot be distinguished in the present data.

In summary, the Fr\'echet mean exhibits finite sample smeariness of considerable magnitude $S_\text{fs} \gg 10$ for more than $5\%$ of the data sets and at least moderate finite sample smeariness of magnitude $S_\text{fs} > 3$ for more than $10\%$ of the data sets. This shows that finite sample smeariness of the spherical mean is a phenomenon which is abundant in data sets describing geomagnetic north pole positions around pole transition periods. In such situations, asymptotic inference on the mean has to be treated with great care.

\begin{figure}[h!]
  \centering
  \subcaptionbox{Data set 141: smeariness for large $k \gg n$}[0.45\textwidth]{\includegraphics[width=0.4\textwidth]{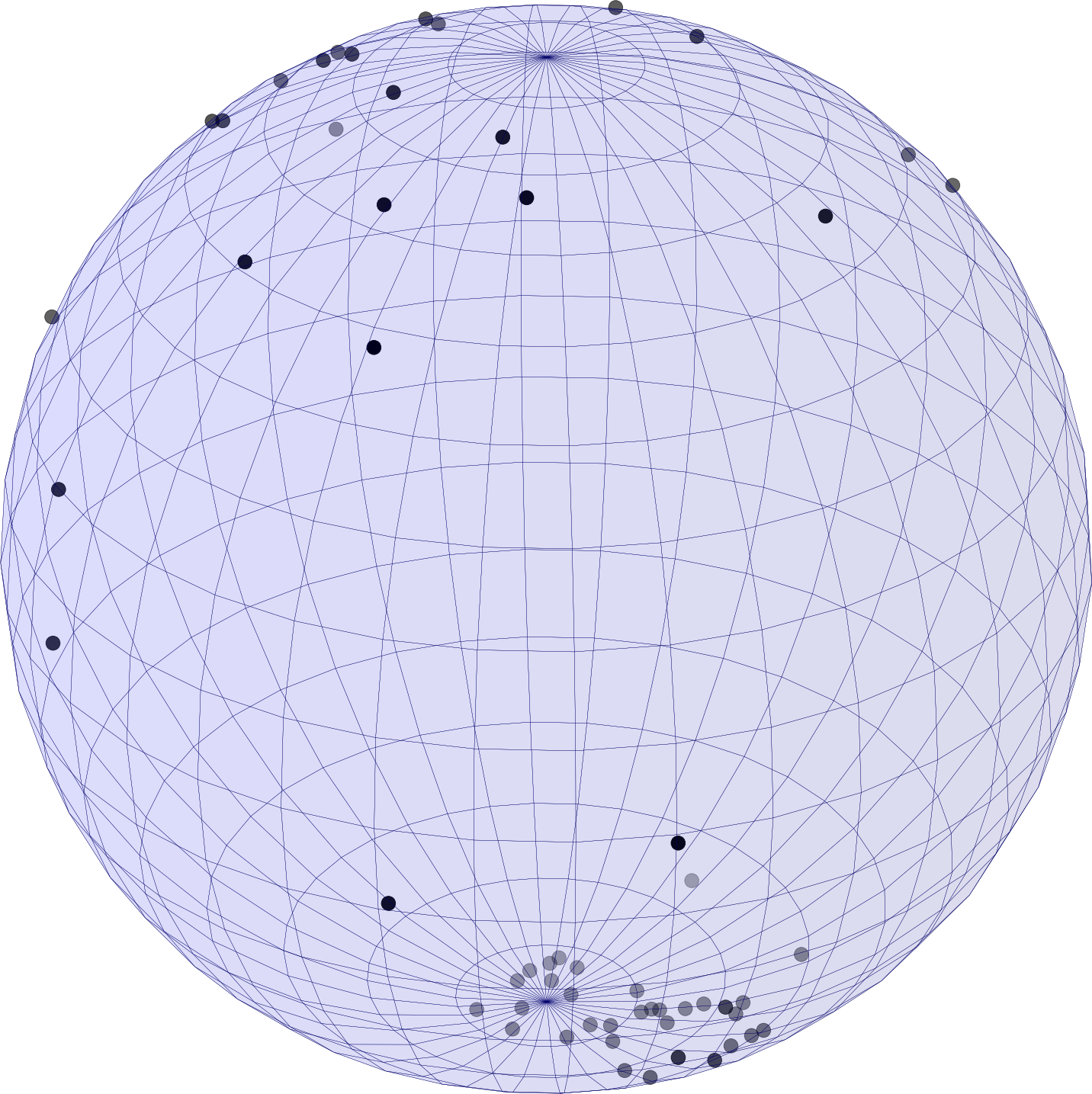}}
  \hspace*{0.02\textwidth}
  \subcaptionbox{Data set 077: smeariness mostly for small $k$}[0.45\textwidth]{\includegraphics[width=0.4\textwidth]{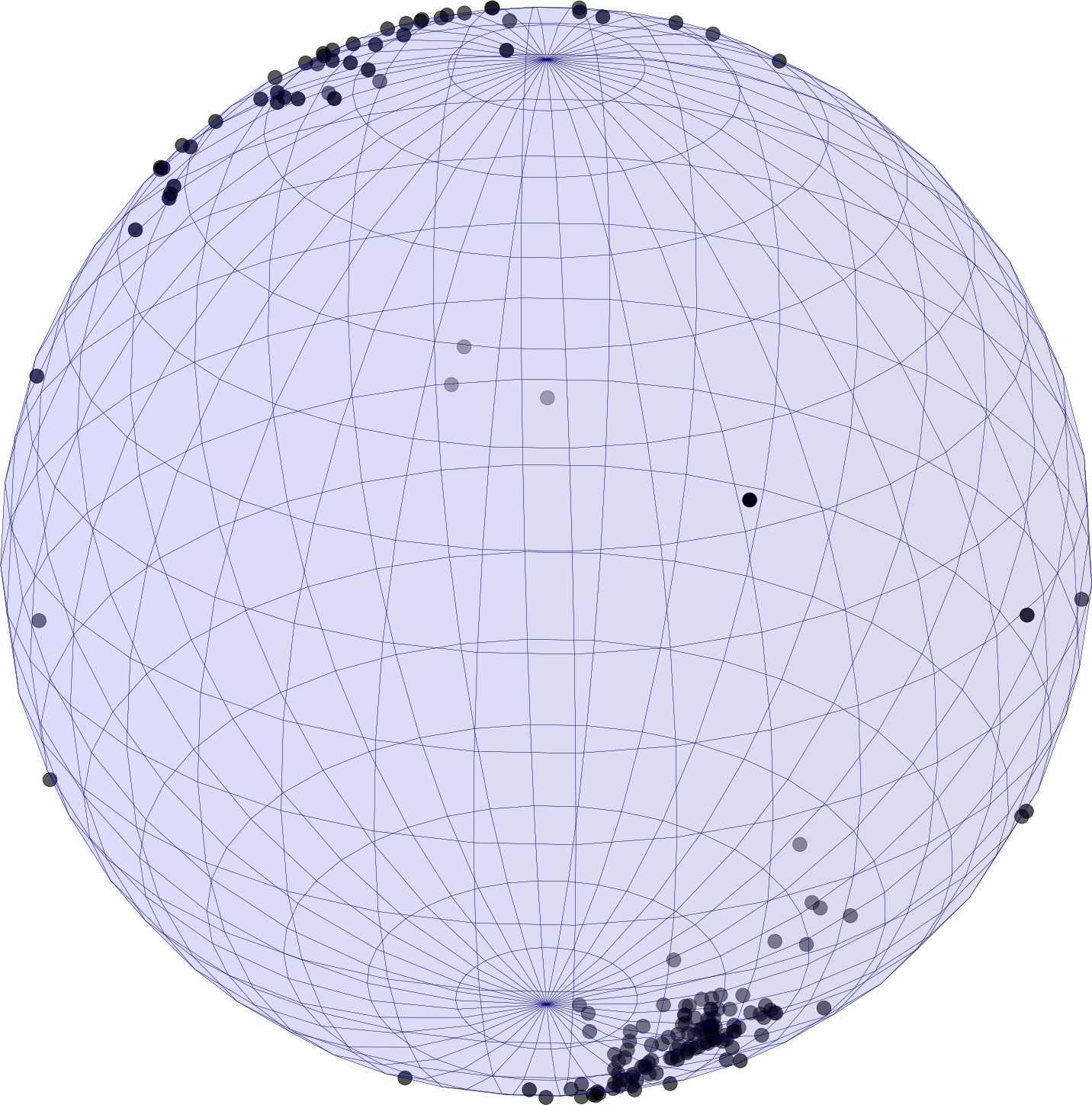}}
  \caption{Examples of data sets which exhibit smeariness. \label{fig:smeary_magnet_examples}}
\end{figure}

\begin{figure}[h!]
  \centering
  \subcaptionbox{Smeariness for large $k$}[0.45\textwidth]{\includegraphics[width=0.45\textwidth]{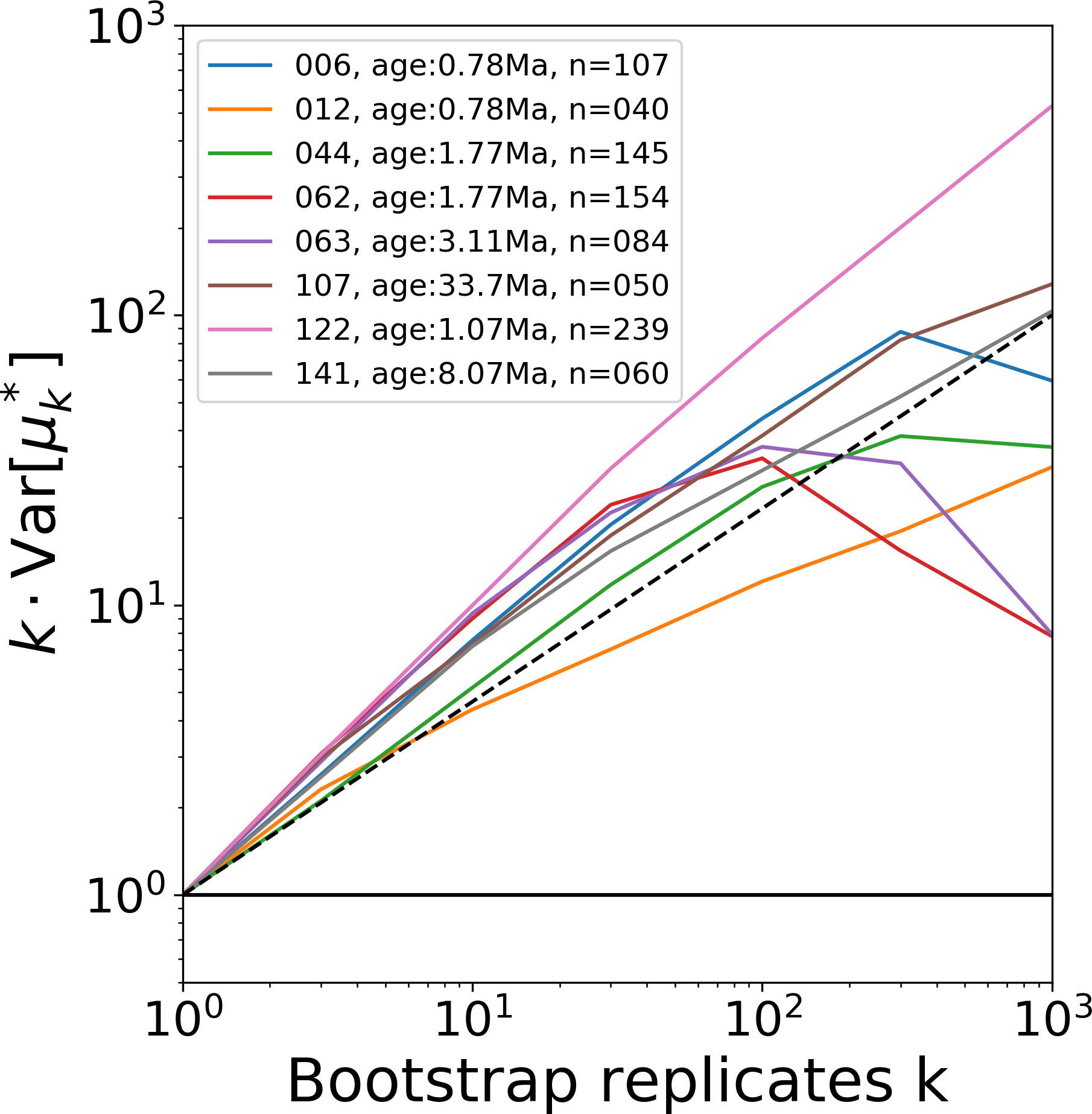}}
  \hspace*{0.02\textwidth}
  \subcaptionbox{Smeariness for small $k$}[0.45\textwidth]{\includegraphics[width=0.45\textwidth]{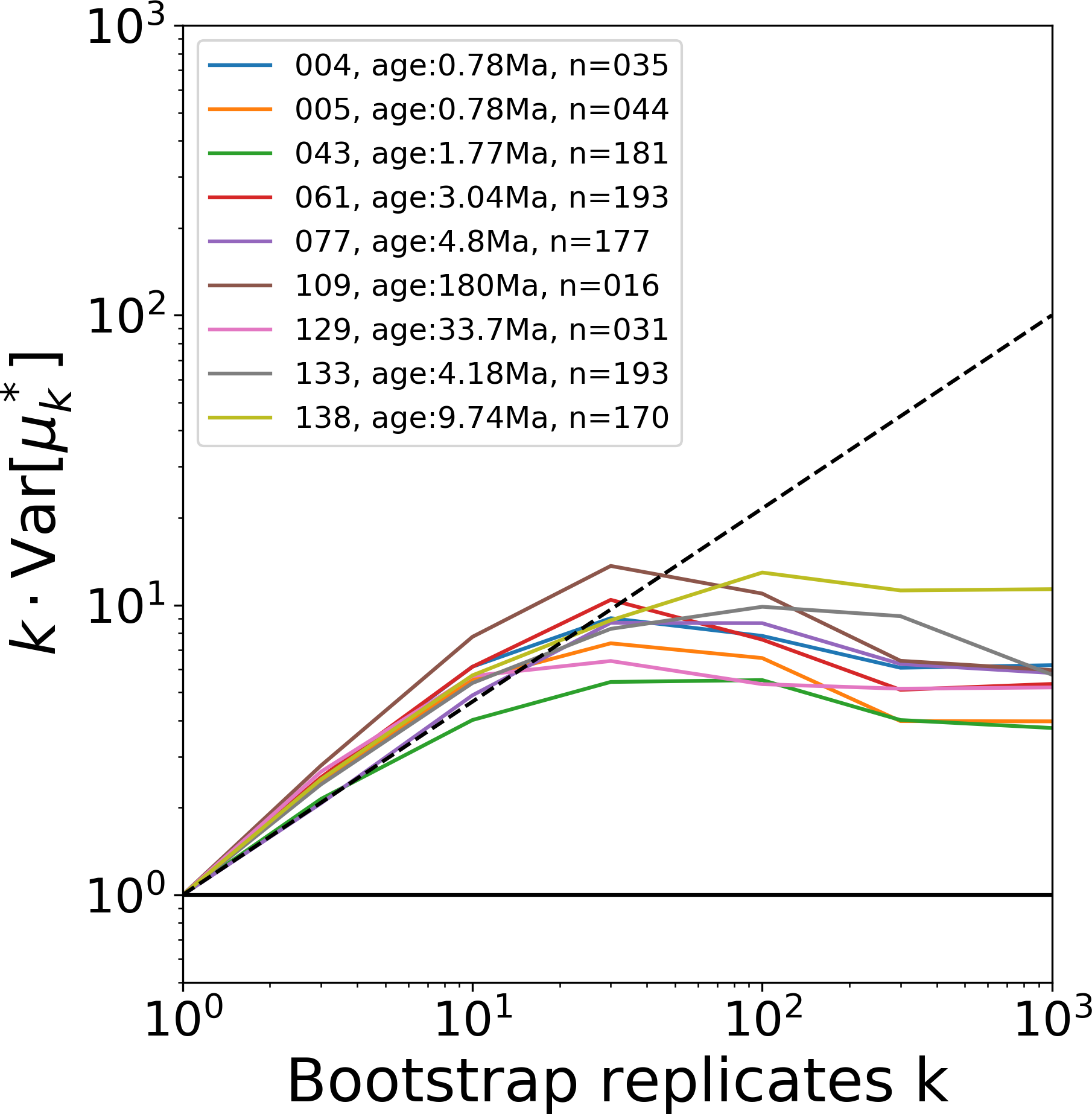}}
  \caption{Bootstrap variances of the mean for samples which exhibit smeariness. (a) Finite sample smeariness is present even at high $k$. (b) Finite sample smeariness is only visible for small $k$. Note that the vertical axis is rescaled by $k$ and we are considering variances, not standard deviations. Therefore, an asymptotic rate of $k^{-1/2}$ is represented by the horizontal solid black line, while a rate of $k^{-1/6}$ is represented by the dashed black line. In general, growing curves indicate higher asymptotic rates than $k^{-1/2}$. \label{fig:smeary_magnets_rates}}
\end{figure}

\section*{Acknowledgments}

The author gratefully acknowledges funding by DFG SFB 755, project B8, DFG SFB 803, project Z2, DFG HU 1575/7 and DFG GK 2088. I am very grateful to Stephan Huckemann for many helpful discussions and detailed comments to the manuscript, Andrew Wood for an inspiring discussion and John Kent and Kanti Mardia for helpful pointers in terms of data application. I would like to thank the anonymous reviewers, whose suggestions helped improve the manuscript.

\appendix

\section{Smeariness with Holes}

In this supplement, we present the details of all calculations needed to prove the results presented in the article. In all cases discussed here, the random variable is invariant under rotation around the polar axis and the mean will always be at the \emph{north pole} $\mu := e_{m+1}$. In consequence, the Fr\'echet function
\begin{align*}
\wF:\mathbb{S}^m \to [0,\infty),~p\mapsto  \int_{\mathbb{S}^m} d^2_{\mathbb{S}^m}(p,q) \,d\Prb^X(q) \, ,
\end{align*}
involving the \emph{squared spherical distance} $d^2_{\mathbb{S}^m}(p,q) = \arccos\langle p,q\rangle^2$ based on the standard inner product $\langle\cdot,\cdot\rangle$ of $\mathbb{R}^{m+1}$, only depends on the polar angle $\psi := \arccos\left< p, \mu \right> \in [0, \pi]$ between the point $p$ and the north pole $\mu$, such that $\psi = 0$ corresponds to $\mu$. This means that there is a function $F:[0, ¸\pi] \to [0,\infty)$ such that
\begin{align*}
\wF (p) = F(\arccos\left< p, \mu \right>) \, .
\end{align*}

Furthermore, due to the rotation symmetry around $\mu$ derivatives of the Fr\'echet function are always diagonal tensors where all diagonal entries are equal. It is therefore sufficient to calculate derivatives $\frac{d^k F}{d\psi^k}$.

In the following, we write the spherical annulus
\begin{align*}
\mathbb{L}_{m,\beta} := \{q\in \mathbb{S}^m: \arccos\left< p, \mu \right> \in [\pi/2, \pi - \beta]\} \, .
\end{align*}
This set is the southern hemisphere with a \emph{hole of radius $\beta$} around the south pole cut out.

\subsection{The Basic Hemisphere Model}

We now recall the key calculation results from \cite{EH19}. Consider a random variable $X$ distributed on the $m$-dimensional unit sphere $\mathbb{S}^m$ ($m\geq 2$) that is uniformly distributed on the lower half sphere $\mathbb{L}_{m,0}$ with total mass $0<\alpha<1$ and assuming the \emph{north pole} $\mu=$ with probability $1-\alpha$.

Setting $\Theta = [-\pi/2,\pi/2]$ and $\Theta^{m-1} \ni \theta = (\theta_1,\ldots,\theta_{m-1})$, defining the functions
\begin{align*}
u : \Theta^{m-1} \to [0,1],~ \theta \mapsto \prod_{j=1}^{m-1}\cos^{m-j}\theta_j \qquad v(\theta) = \prod_{j=1}^{m-1}\cos\theta_j\,,
\end{align*}
we have the spherical volume element $u(\theta)\,d\theta\,d\phi$. The volume of the full $\mathbb{S}^m$ is given by
\begin{align*}
V_m = {\vol}(\mathbb{S}^m)= \frac{2\pi^{\frac{m+1}{2}}}{\Gamma\left(\frac{m+1}{2}\right)} \, .
\end{align*}
Using suitable coordinates, \cite{EH19} show that
\begin{align*}
F(\psi) =& \psi^2(1-\alpha) + F(0)- \frac{4\pi\alpha}{V_m}  \int_{\Theta^{m-1}}\limits \, u(\theta)\,\int^{\psi}_{0} \limits  \arcsin\big(v(\theta)\,\sin\phi\big) \, d\phi \, d \theta \, .
\end{align*}

Consider the derivatives
\begin{align*}
F'(\psi) &= 2\psi(1-\alpha) - \frac{4\pi\alpha}{V_m} \int_{\Theta^{m-1}} \limits  u(\theta) \,  \arcsin\big(v(\theta)\,\sin\psi\big) \, d \theta \,,\nonumber\\
F''(\psi) &= 2(1-\alpha) - \frac{4\pi\alpha}{V_m} \int_{\Theta^{m-1}} \limits  u(\theta)\,v(\theta) \,  \frac{\cos\psi}{\sqrt{1-v(\theta)^2\,\sin^2\psi}} \, d \theta\\
F^{(3)}(\psi) &= \frac{4\pi\alpha}{V_m}  \int_{\Theta^{m-1}} \limits  u(\theta)\,v(\theta) \,  \frac{(1-v(\theta)^2)\sin\psi}{\big(1-v(\theta)^2\,\sin^2\psi\big)^{3/2}} \, d \theta\\
F^{(4)}(\psi) &= \frac{4\pi\alpha}{V_m}  \int_{\Theta^{m-1}} \limits  u(\theta)\,v(\theta)\,(1-v(\theta)^2) \,  \frac{(1+2v(\theta)^2\sin^2\psi)\cos\psi}{\big(1-v(\theta)^2\,\sin^2\psi\big)^{5/2}}\, d \theta \, .
\end{align*}

By direct calculation, we can conclude
\begin{Lem}
  \begin{align}
  F''(\psi) &\geq 0 & \textnormal{for } & \psi \in [0, \pi) \, , \label{eq:nohole-der2}\\
  F^{(3)}(\psi) &> 0 & \textnormal{for } & \psi \in (0, \pi) \, ,  \label{eq:nohole-der3}\\
  F^{(4)}(\psi) &\geq F^{(4)}(0) \cos\psi > 0 & \textnormal{for } & \psi \in [0, \pi/2) \, ,  \nonumber\\
  F^{(4)}(\psi) &\ge \frac{1}{2} F^{(4)}(0) & \textnormal{for } & \psi \in [0, \pi/3) \, . \nonumber
  \end{align}
\end{Lem}

%\begin{align}
%&\forall \psi \in [0,\pi] & F''(\psi) &\geq 2(1-\alpha) - \frac{4\pi\alpha}{V_m} \int_{\Theta^{m-1}} \limits  u(\theta)\,v(\theta) \, d \theta \label{eq:nohole-der2}\\
%&&&= 2 - \alpha \left(2 + \frac{V_{m+1}}{V_m}\right) = 0 \nonumber\\
%&\forall \psi \in (0,\pi) & F^{(3)}(\psi) &\geq \sin\psi \frac{4\pi\alpha}{V_m} \int_{\Theta^{m-1}} \limits u(\theta)\,v(\theta) \, (1-v(\theta)^2) \, d \theta > 0 \label{eq:nohole-der3}\\
%&\forall \psi \in [0,\pi/2) & F^{(4)}(\psi) &\geq \cos\psi \frac{4\pi\alpha}{V_m} \int_{\Theta^{m-1}} \limits u(\theta)\,v(\theta) \, (1-v(\theta)^2) \, d \theta \label{eq:nohole-der4}\\
%&&&= F^{(4)}(0) \cos\psi > 0 \nonumber\\
%& \forall \psi \in [0,\pi/3] & F^{(4)}(\psi) &\geq F^{(4)}(0) \cos\psi \ge \frac{1}{2} F^{(4)}(0) \label{eq:nohole-der4-strictly-pos} \, .
%\end{align}

We recall $F^{(4)}(0) =\frac{\alpha V_{m+1}}{V_m}\,\frac{m-1}{m+2}=c_m>0$ and that the inequality in \eqref{eq:nohole-der2} is strict for $\psi \neq 0$, due to \eqref{eq:nohole-der3}. Hence we infer that $F'(\psi)$ is strictly increasing in $\psi$ from $F'(0)=0$, yielding that there is no minimum of $F$ other than $\psi = 0$, which corresponds to $p=\mu$.

\subsection{Rotation Symmetric Random Variables}

We use coordinates
\begin{align*}
q &=\begin{pmatrix}q_1\\q_2\\\vdots\\q_{m-1}\\q_{m}\\q_{m+1}\end{pmatrix} = \begin{pmatrix} \sin \theta\,\sin\phi\,\left(\prod_{j=3}^{m-1} \sin \theta_j\right) \sin \theta_m \\ \sin \theta\,\sin\phi\,\left(\prod_{j=3}^{m-1} \sin \theta_j\right) \cos \theta_m \\ \vdots \\ \sin \theta\,\sin\phi\,\cos\theta_3 \\ \sin \theta\,\cos\phi\ \\ \cos\theta \end{pmatrix}\, , &
p &=\begin{pmatrix}p_1\\\vdots\\p_{m-1}\\p_{m}\\p_{m+1}\end{pmatrix} = \begin{pmatrix} 0 \\ \vdots \\ 0 \\ \sin\psi \\ \cos\psi \end{pmatrix}
\,,
\end{align*}
where $\mu = e_{m+1} = (0,\ldots,0)$ in these coordinates. Note that we exploit rotation symmetry here to define $p$ in especially simple way that eliminates all $\theta_k$ with $k \ge 3$ from all following calculations.

Now we have introduced the necessary notation to prove Theorem 2.14.

\subsubsection{Proof of Theorem 2.14} \label{sec:no-crit-proof}

Consider a probability measure with a point mass at the north pole with weight $1-\alpha$ and a uniform distribution with weight $\alpha$ in a ball of radius $\delta$ around the south pole. Then, using notation $h(\psi, \theta, \phi) := \cos\psi \cos\theta + \sin\psi \sin\theta \cos\phi$ and $I_m := \int_0^\pi \sin^{m}\phi \, d\phi$ the Fr\'echet function is

\begin{align*}
F (\alpha, \delta, \psi) :=& (1- \alpha) \psi^2 + \alpha g(\delta) \int_{\pi- \delta}^{\pi} \sin^{m-1}\theta \int_0^\pi \sin^{m-2}\phi \, \Big( \arccos h(\psi, \theta, \phi) \Big)^2 \, d\phi \, d\theta\\
g(\delta) :=& \left( \int_{\pi - \delta}^{\pi} \sin^{m-1}\theta \int_0^\pi \sin^{m-2}\phi \, d\phi \, d\theta \right)^{-1} = \left( I_{m-2}\int_{\pi - \delta}^{\pi} \sin^{m-1}\theta \, d\theta \right)^{-1}\, . \nonumber
\end{align*}

No note that due to convexity of the function $x \mapsto ( \arccos x )^2$,
\begin{align*}
\Big( \arccos h(\psi, \theta, \phi) \Big)^2 \ge& \Big( \arccos h(0, \theta, \phi) \Big)^2 - 2 \frac{\arccos h(0, \theta, \phi)}{\sqrt{1 - h(0, \theta, \phi)^2}} \Big(h(\psi, \theta, \phi) - h(0, \theta, \phi) \Big)\\
=& \theta^2 - 2 \frac{\theta}{\sin\theta} \Big((\cos\psi-1) \cos\theta + \sin\psi \sin\theta \cos\phi \Big)\, .
\end{align*}

Thus, using $\int_0^\pi \cos\phi \sin^{m-2}\phi \, d\phi = 0$ we get
\begin{align*}
\int_0^\pi \sin^{m-2}\phi \, \Big( \arccos h(\psi, \theta, \phi) \Big)^2 \, d\phi \ge& \int_0^\pi \sin^{m-2}\phi \, \Big( \theta^2 - 2 \frac{\theta \cos\theta}{\sin\theta} (\cos\psi-1) \Big) \, d\phi\\
=& I_{m-2} \Big( \theta^2 + 2 \frac{\theta \cos\theta}{\sin\theta} (1 - \cos\psi) \Big)
\end{align*}

This leads to the lower bound
\begin{align*}
F (\alpha, \delta, \psi) \ge& (1- \alpha) \psi^2 + \alpha g(\delta) I_{m-2} \int_{\pi- \delta}^{\pi} \Big( \theta^2 + 2 \frac{\theta \cos\theta}{\sin\theta} (1 - \cos\psi) \Big) \sin^{m-1}\theta \, d\theta\\
=& F (\alpha, \delta, 0) + (1- \alpha) \psi^2 + 2 \alpha g(\delta) I_{m-2} (1 - \cos\psi) \int_{\pi- \delta}^{\pi} \theta \cos\theta\sin^{m-2}\theta \, d\theta\\
=& F (\alpha, \delta, 0) + (1- \alpha) \psi^2 + \frac{2 \alpha g(\delta) I_{m-2}}{m-1} (1 - \cos\psi) \left( -(\pi-\delta) \sin^{m-1}\delta - \int_{\pi- \delta}^{\pi} \sin^{m-1}\theta \, d\theta \right)\\
\ge& F (\alpha, \delta, 0) + (1- \alpha) \psi^2 - \frac{\alpha g(\delta) I_{m-2}}{m-1} \psi^2 \left( (\pi-\delta) \sin^{m-1}\delta + \int_{\pi- \delta}^{\pi} \sin^{m-1}\theta \, d\theta \right)\\
\ge& F (\alpha, \delta, 0) + (1- \alpha) \psi^2 - \frac{\alpha}{m-1} \psi^2 \left( \frac{m(\pi-\delta) \sin^{m-1}\delta}{\sin^m\delta} + 1 \right)
\end{align*}

It is clear that $F (\alpha, \delta, \psi)$ takes its global minimum at $\psi = 0$ if
\begin{align*}
(1- \alpha) - \frac{\alpha}{m-1} \left( \frac{m(\pi-\delta)}{\sin\delta} + 1 \right) &> 0\\
\alpha \frac{m(\pi+\sin\delta -\delta)}{(m-1)\sin\delta} &< 1\\
\alpha &< \frac{(m-1)\sin\delta}{m(\pi+\sin\delta -\delta)}
\end{align*}

We pick $\alpha = \frac{\sin\delta}{4\pi}$, which satisfies the above inequality. Therefore, we have a non-smeary Fr\'echet mean at the north pole for any $\delta$, if we choose this $\alpha$. The value of the probability density at the south pole can be lower bounded for $\delta \le \pi /2$ by
\begin{align*}
\alpha g(\delta) = \frac{\sin\delta}{4\pi}g(\delta) \ge \frac{m \sin\delta}{4\pi I_{m-2} \delta^m} \ge \frac{m}{4\pi^2 I_{m-2} \delta^{m-1}} \, .
\end{align*}
As $\delta \to 0$ the density diverges, thus any arbitrarily high probability density can be achieved at the south pole by choosing a suitable $\delta$ and $\alpha = \frac{\sin\delta}{4\pi}$. Note that this result holds for any $m \ge 2$, since we do not need to differentiate under the integrals here because of the simplification achieved by the shown lower bound. \qed

\subsection{Derivatives of the Fr\'echet Function}

To keep the calculations readable, we introduce some shorthand notation
\begin{align*}
h(\psi, \theta, \phi) &:= \cos\psi \cos\theta + \sin\psi \sin\theta \cos\phi\\
h'(\psi, \theta, \phi) &:= \frac{\partial}{\partial\psi} h(\psi, \theta, \phi) = -\sin\psi \cos\theta + \cos\psi \sin\theta \cos\phi\\
a(\psi, \theta, \phi) &:= \arccos h(\psi, \theta, \phi)\\
s(\theta, \phi) &:= \sin\theta \sin\phi \, ,
\end{align*}
where we will suppress the arguments in the following, and we note
\begin{align*}
h''(\psi, \theta, \phi) := \frac{\partial^2 h}{\partial\psi^2} = -h \qquad \text{and} \qquad 1 - h^2 = (h')^2 + s^2 \, .
\end{align*}

In the following calculations we use
\begin{align*}
h(0, \theta, \phi) &= \cos\theta &
h'(0, \theta, \phi) &= \sin\theta \cos\phi\\
a(0, \theta, \phi) &:= \theta &
s(\theta, \phi) &:= \sin\theta \sin\phi \, .
\end{align*}
With the notation $I_m := \int_0^\pi \sin^{m}\phi \, d\phi$, we have $I_{m-2} = \frac{m}{m-1} I_{m}$.

Since we restrict attention to rotation invariant random variables, we first consider on uniform distribution on the subsphere $\mathbb{S}^{m-1}_\psi$ given by fixed polar angle $\psi$. The Fr\'echet function of a random variable with this distribution can be written as
\begin{align*}
F_\theta (\psi) :=& g (\theta) \sin\theta \int_0^\pi s^{m-2} \, a^2 \, d\phi \, , & g(\theta) :=& \left( \sin\theta \int_0^\pi s^{m-2} \, d\phi \right)^{-1} \, . 
\end{align*}
and we can calculate derivatives, writing $f_j(\theta,\psi) := \frac{1}{2 g(\theta)} \frac{\partial^j F_\theta}{\partial\psi^j}$
\begin{align*}
f_1(\theta,\psi) =& \sin\theta \int_0^\pi s^{m-2} \left( \frac{-h' \, a}{\big(1-h^2\big)^{1/2}} \right) \, d\phi\\
f_2(\theta,\psi) =& \sin\theta \int_0^\pi s^{m-2} \left(\frac{(h')^2}{1-h^2} + \frac{h \, s^2 \, a}{\big(1-h^2\big)^{3/2}}\right) \, d\phi\\
f_3(\theta,\psi) =& \sin\theta \int_0^\pi s^{m} \left(\frac{-3 h\, h'}{\big(1-h^2\big)^2} + \frac{(1+2h^2) \, h' \, a}{\big(1-h^2\big)^{5/2}} \right) \, d\phi\\
f_4(\theta,\psi) =& \sin\theta \int_0^\pi s^{m} \left(\frac{3s^2 \, h^2 - 4(1+2h^2)(h')^2\vphantom{\Big(}}{\big(1-h^2\big)^3} \right.\\
&+ \left. \frac{\Big(4 (2+h^2)(h')^2 - s^2(1+2h^2) \Big) h \, a}{\big(1-h^2\big)^{7/2}}  \right) \, d\phi \, .
\end{align*}

Above, we differentiate under the integral. In Lemma \ref{lem:derivatives-integral}, we show that for sufficiently high dimension the derivatives with respect to $\psi$ can be interchanged with the integrals over $\theta$ and $\phi$.

\begin{Lem} \label{lem:derivatives-integral}
  We can differentiate under the integral in the following sense, for arbitrary integral bounds of the $\theta$-integral in $[0,\pi]$
  \begin{align*}
  \frac{\partial^2}{\partial\psi^2}\int \sin\theta \int_0^\pi s^{m-2} \, a^2 \, d\phi \, d\theta &= 2 \int f_2(\theta,\psi) d \theta & \text{for } m &\ge 3\\
  \frac{\partial^3}{\partial\psi^3}\int \sin\theta \int_0^\pi s^{m-2} \, a^2 \, d\phi \, d\theta &= 2 \int f_3(\theta,\psi) d \theta & \text{for } m &\ge 4\\
  \frac{\partial^4}{\partial\psi^4}\int \sin\theta \int_0^\pi s^{m-2} \, a^2 \, d\phi \, d\theta &= 2 \int f_4(\theta,\psi) d \theta & \text{for } m &\ge 5 \, .
  \end{align*}
  If for an arbitrarily small $\eps >0$, one restricts to $\theta \in [0, \pi - 2 \eps]$ and $\psi \in [0, \eps]$, the bounds on the dimension $m$ in these equations can be lowered by one to $m \ge 2$, $m \ge 3$ and $m \ge 4$, respectively.
\end{Lem}

\begin{proof}
  For the assertion to hold, it suffices to show, that the $f_j(\theta,\psi)$ are integrable for the respective values of $m$. Since the numerators can all be easily bounded, the only problem is to bound the denominators under the integrals. Recall that $1- h^2 = (h')^2 + s^2$ and use
  \begin{align} \label{eq:smeary-estimates}
  \frac{|h'|}{\big((h')^2 + s^2\big)^{1/2}} \le 1 \qquad \text{and} \qquad \frac{s}{\big((h')^2 + s^2\big)^{1/2}} \le 1
  \end{align}
  thus we get
  \begin{align*}
  \left| \sin\theta \int_0^\pi s^{m-2} \left(\frac{(h')^2}{1-h^2} + \frac{h \, s^2 \, a}{\big(1-h^2\big)^{3/2}} \right) \, d\phi \right| &\le (\pi+1) \sin\theta \int_0^\pi s^{m-3} \, d\phi\\
  \left|\sin\theta \int_0^\pi s^{m} \left(\frac{-3 h\, h'}{\big(1-h^2\big)^2} + \frac{(1+2h^2) \, h' \, a}{\big(1-h^2\big)^{5/2}} \right) \, d\phi \right| &\le 3(\pi+1) \sin\theta \int_0^\pi s^{m-4} \, d\phi\\
  \left| \sin\theta \int_0^\pi s^{m} \left(\frac{3s^2 \, h^2 - 4(1+2h^2)(h')^2}{\big(1-h^2\big)^3} \right.\right.&\\
  + \left.\left. \frac{\Big(4 (2+h^2)(h')^2 - s^2(1+2h^2) \Big) h \, a}{\big(1-h^2\big)^{7/2}} \right) \, d\phi \right| &\le 15(\pi+1) \sin\theta \int_0^\pi s^{m-5} \, d\phi \, .
  \end{align*}
  We see that these bounds are finite for the required dimensions.
  
  To see that the final claim holds, note that the function $x \mapsto \arccos^2(x)$ is $C^\infty$ on $(-1,1]$ and $h$ is $C^\infty$ in the whole domain. Since $\lim_{x \to 1} \frac{\arccos(x)}{\sqrt{1-x^2}} = 1$, the poles at $h = 1$ in the integrals are one order lower than the poles at $h = -1$. Since $h = -1$ only holds if $\theta = \pi - \psi$ and $\phi = \pi$, these poles can be excluded by restricting to $\theta + \psi \le \pi - \eps$ for some $\eps > 0$. Then, $K(\eps) := \frac{\arccos h(\psi, \pi - \psi - \eps, -\pi)}{\sqrt{1 - h^2(\psi, \pi - \psi - \eps, -\pi)}}$ replaces $\pi$ on the right hand side of the inequalities above and the powers of $s$ are higher by one.
\end{proof}

Since we will be interested in random variables, which exhibit a local minimum of the Fr\'echet function at the north pole $\theta=0$, we first consider $f_j(\theta,0)$. We note that
\begin{align*}
f_1(\theta, 0) = f_3(\theta, 0) = 0 \, ,
\end{align*}
so we can restrict attention to $f_2(\theta,0)$ and $f_4(\theta,0)$.
\begin{align}
f_2(\theta,0) =& \sin^{m-1}\theta \int_0^\pi \limits \sin^{m-2}\phi \left(\cos^2\phi + \frac{\theta\,\cos\theta\,\sin^2\phi}{\sin\theta} \right) \, d\phi \nonumber\\
=& \sin^{m-1}\theta \left(I_{m-2} - I_{m} + \frac{\theta\,\cos\theta}{\sin\theta} I_{m} \right)\nonumber\\
=& \sin^{m-2}\theta \left(\frac{1}{m-1}\sin\theta + \theta\,\cos\theta \right) I_{m}\nonumber\\
=& I_{m} \frac{1}{m-1} \frac{d}{d\theta} \left( \theta\,\sin^{m-1}\theta \right)\label{eq:f2} \, .
\end{align}

\begin{figure}[h!]
  \centering
  \includegraphics[width=0.8\textwidth]{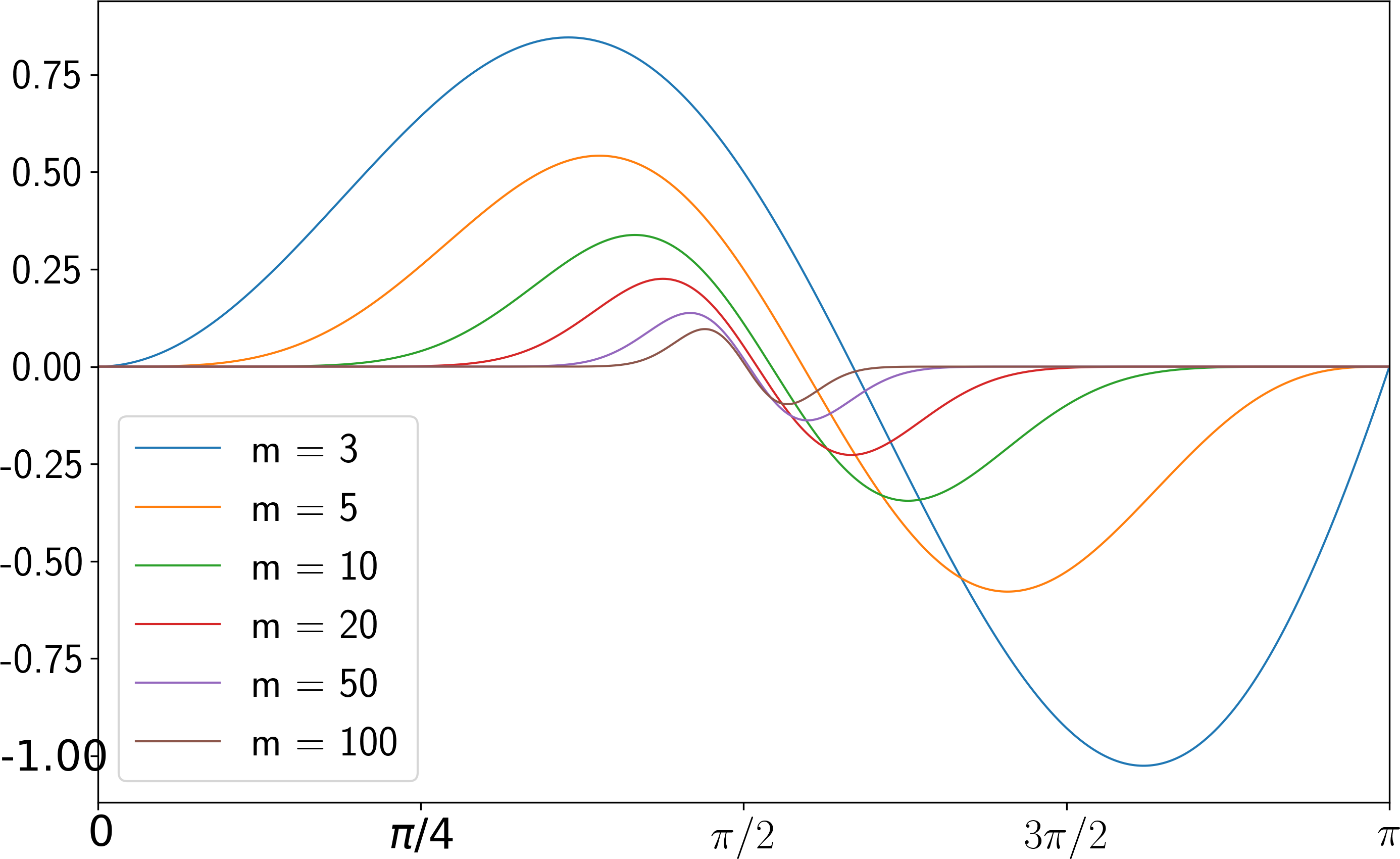}
  \caption{\it Plots of $\theta \mapsto f_2(\theta,0) = \frac{1}{2 g(\theta)} \frac{\partial^2 F_\theta}{\partial\psi^2}|_{\psi=0}$ for different dimension $m$. One can clearly see that for $m \to \infty$ the position of the zero approaches $\theta = \pi/2$ from above. \label{fig:2nd_der}}
\end{figure}

%\subsubsection{Calculation of $f_4(\theta,0)$}
\begingroup
\setlength\abovedisplayskip{-\baselineskip}
\begin{align}
f_4(\theta,0) =& \sin^{m-3}\theta \int_0^\pi \sin^{m}\phi (3 \cos^2\theta\,\sin^{2}\phi - 4(1+2\cos^2\theta)\cos^2\phi) \, d\phi\nonumber\\
&+ \sin^{m-4}\theta \int_0^\pi \sin^{m}\phi \Big(4 (2+\cos^2\theta)\cos^2\phi - (1+2\cos^2\theta)\sin^{2}\phi \Big) \theta \, \cos\theta \, d\phi\nonumber\\
=& \sin^{m-3}\theta \int_0^\pi \sin^{m}\phi \Big( (4 + 11 \cos^2\theta)\,\sin^{2}\phi - (4+8\cos^2\theta) \Big) \, d\phi\nonumber\\
&+ \sin^{m-4}\theta \int_0^\pi \sin^{m}\phi \Big( -(9 + 6 \cos^2\theta)\,\sin^{2}\phi + (8+4\cos^2\theta)\Big) \theta \, \cos\theta \, d\phi\nonumber\\
=& I_{m} \sin^{m-3}\theta \Big( \frac{m+1}{m+2} (15- 11 \sin^2\theta) - (12 - 8\sin^2\theta) \Big)\nonumber\\
&- I_{m} \theta \, \cos\theta \sin^{m-4}\theta \Big(\frac{m+1}{m+2} (15 - 6 \sin^2\theta) - (12 - 4\sin^2\theta) \Big)\nonumber\\
=& \frac{I_{m}}{m+2} \sin^{m-3}\theta \Big( (3m-9) - (3m-5) \sin^2\theta \Big)\nonumber\\
&- \frac{I_{m}}{m+2} \theta \, \cos\theta \sin^{m-4}\theta \Big( (3m-9) - (2m-2) \sin^2\theta \Big)\nonumber\\
%  =&\frac{I_{m}}{m+2} \sin^{m-3}\theta \Big( (3m-9) - (3m-5) \sin^2\theta \Big)\\
%  &- \frac{I_{m}}{m+2} \theta\,\frac{d}{d\theta} \Big(3\sin^{m-3}\theta - 2 \sin^{m-1}\theta\Big)\\
=& \frac{I_{m}}{m+2} \Big( (3m-6) \sin^{m-3}\theta - (3m-3) \sin^{m-1}\theta \Big)\nonumber\\
&- \frac{I_{m}}{m+2} \frac{d}{d\theta} \Big(3\theta\,\sin^{m-3}\theta - 2 \theta\,\sin^{m-1}\theta\Big)\label{eq:f4} \, .
\end{align}
\endgroup

\begin{figure}[h!]
  \centering
  \includegraphics[width=0.8\textwidth]{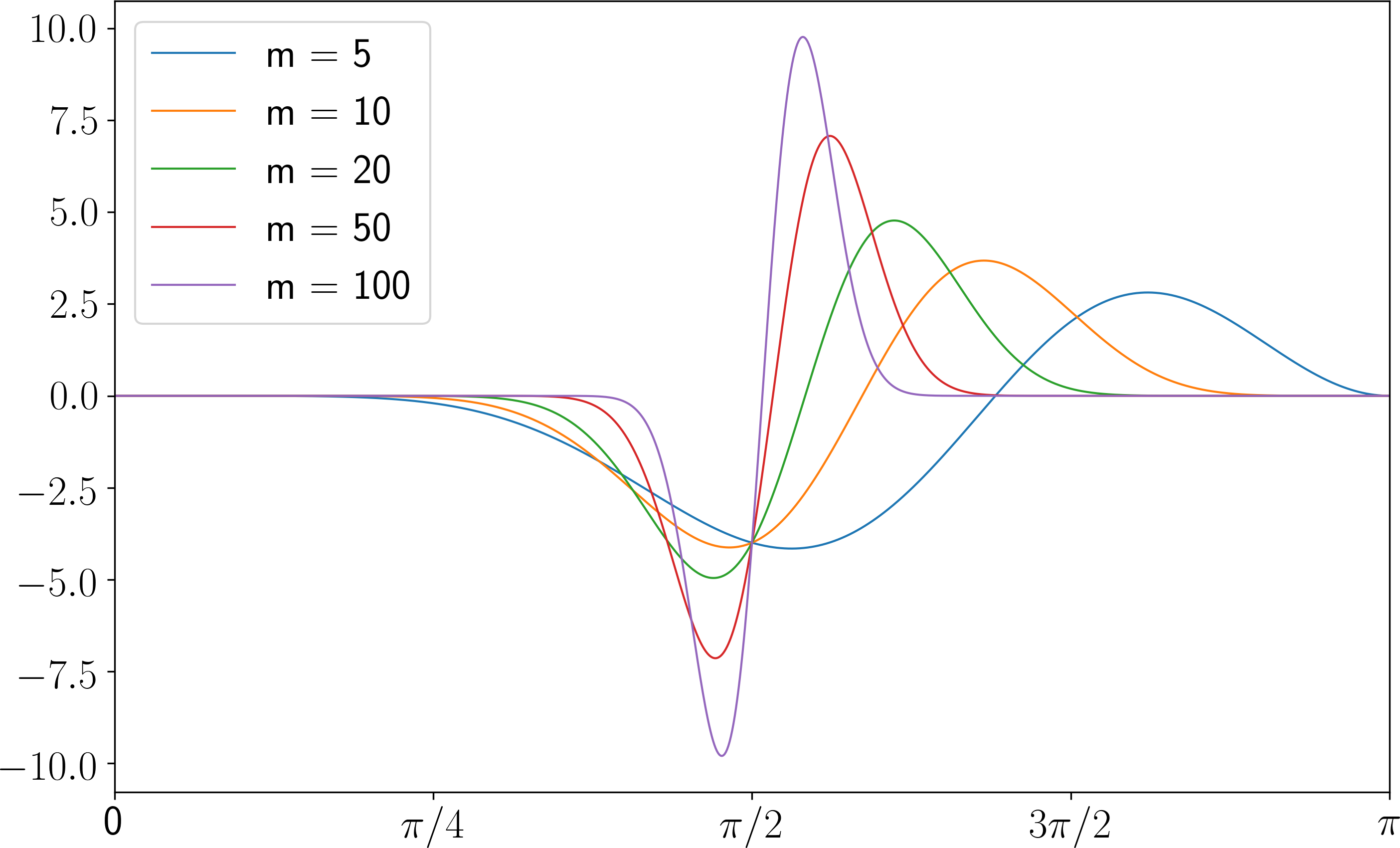}
  \caption{\it Plots of $\theta \mapsto f_4(\theta,0) = \frac{1}{2 g(\theta)} \frac{\partial^4 F_\theta}{\partial\psi^4}|_{\psi=0}$ for different dimension $m$. One can clearly see that for $m \to \infty$ the lower bound of the region where the fourth derivative is positive approaches $\theta = \pi/2$ from above.\label{fig:4th_der}}
\end{figure}

The function $\theta \mapsto f_2(\theta,0)$ is plotted for several $m$ in Figure \ref{fig:2nd_der} and $\theta \mapsto f_4(\theta,0)$ is given in Figure \ref{fig:4th_der}. From equation \eqref{eq:f2} we can see that $f_2(\theta,0)$ starts out positive at $\theta$ near $0$ and has exactly one sign change at $\frac{1}{m-1}\sin\theta + \theta\,\cos\theta = 0$. We denote the position of the zero by $\theta_{m,2}$. From equation \eqref{eq:f4} we can see that $f_4(\theta,0)$ starts out negative at $\theta$ near $0$ and has exactly one sign change at a point we call $\theta_{m,4}$. Furthermore, the Figures show that the contributions to both the second and fourth derivative are increasingly pronounced close to the equator with increasing dimension. Note that $f_4(\pi/2,0) = -4$ independent of dimension, as can be seen in Figure \ref{fig:4th_der}, so we immediately see $\theta_{m,4} > \frac{\pi}{2}$.

\begin{figure}[h!]
  \centering
  \includegraphics[width=0.8\textwidth]{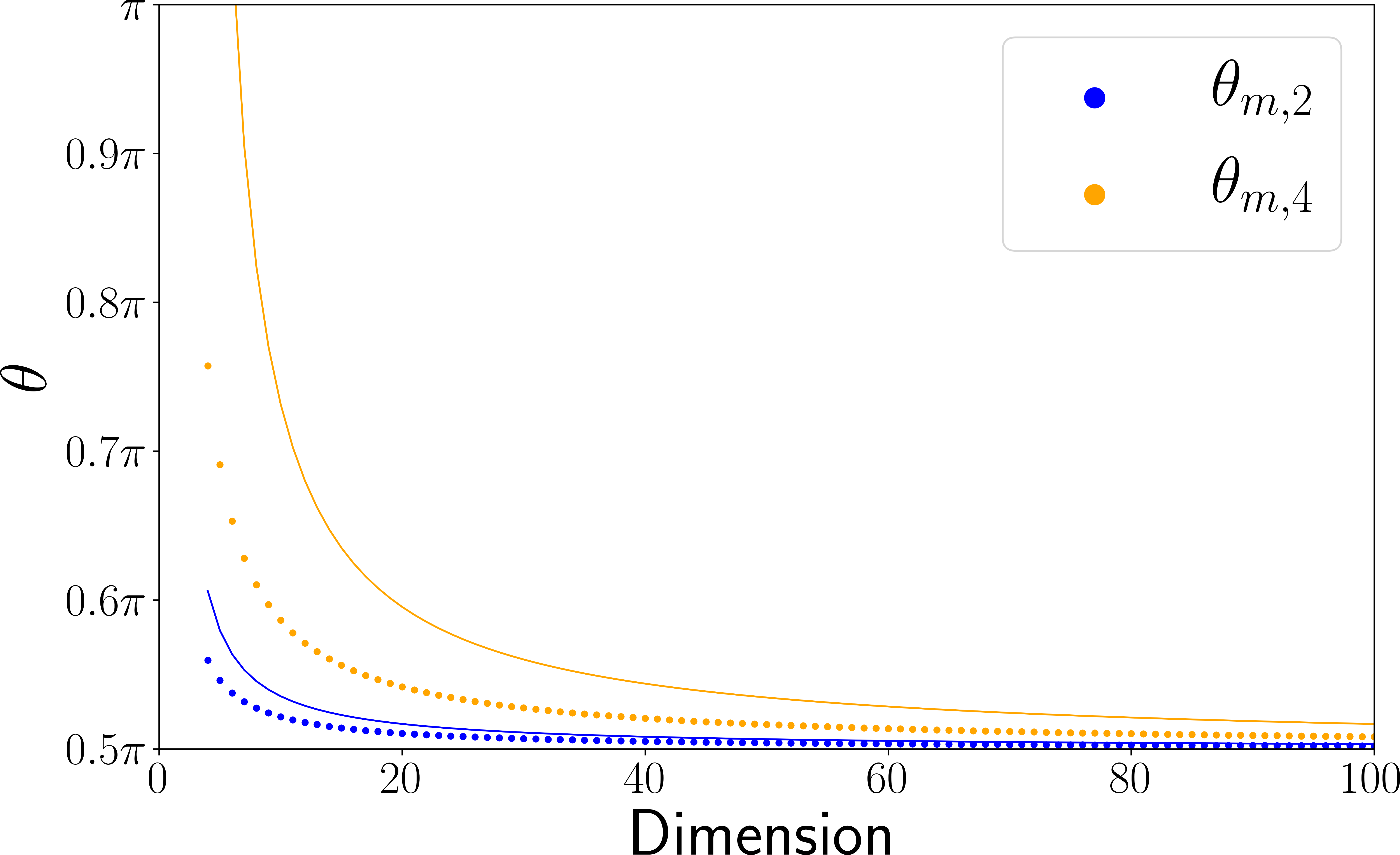}
  \caption{\it Numerically determined values for $\theta_{m,2}$ and $\theta_{m,4}$ for $m \le 100$. One can clearly see that the values approach $\pi/2$ from above. \label{fig:deltas_ring}}
\end{figure}

%\begin{Lem}
%  In supplement A.4 we show the bounds
%  \begin{align*}
%    \frac{\pi}{2} + \frac{1}{3(m-1)} &\le \theta_{m,2} \le \frac{\pi}{2} + \frac{1}{m-1}\\
%    \frac{\pi}{2} + \frac{1}{3(m-1)} &\le \theta_{m,2} \le \theta_{m,4} \le \frac{\pi}{2} + \frac{16}{\pi(m-3)}
%  \end{align*}
%  which means that both $\theta_{m,2}$ and $\theta_{m,4}$ approach $\frac{\pi}{2}$ from above with a rate of $m^{-1}$. This can be interpreted as follows. Adding a point mass at the north pole, there are random variables for which the Fr\'echet function has a minimum at the north pole $\mu=e_{m+1}$ with vanishing Hessian and positive fourth derivative, whose range goes only slightly beyond $\theta_{m,2}$ and $\theta_{m,4}$. Consequently, in higher dimension, the random variables can range over the northern hemisphere and only a small region across the equator in the southern hemisphere while the mean is still smeary. This is an especially egregious example of the curse of dimensionality.
%\end{Lem}

\begin{Lem}
  The position $\theta_{m,2}$ of the zero of $\theta \mapsto f_2(\theta,0)$ is bounded by
  \begin{align*}
  \frac{\pi}{2} + \frac{1}{3(m-1)} &\le \theta_{m,2} \le \frac{\pi}{2} + \frac{1}{m-1}\, .
  \end{align*}
\end{Lem}

\begin{proof}
  This is positive while $\frac{1}{m-1}\sin\theta + \theta\,\cos\theta > 0$. The point where it switches sign is determined by $\theta = - \frac{1}{m-1}\tan\theta$, where $\theta > \pi/2$. Using $\theta = \pi/2 + \delta$ and $1 - \delta^2/2 \le \cos\delta \le 1 - \delta^2/3$ and $\delta/2 \le \sin\delta \le \delta$ , which hold on $[0,\pi/2]$
  \begin{align*}
  \frac{1}{m-1}\sin\theta + \theta\,\cos\theta < 0 \quad &\Leftrightarrow \quad \frac{1}{m-1}\cos\delta - \left( \frac{\pi}{2} + \delta \right)\,\sin\delta < 0 \\
  \Leftarrow \quad \frac{1 - \delta^2/3}{m-1} - \left( \frac{\pi}{2} + \delta \right)\delta/2 < 0 \quad &\Leftrightarrow \quad  - \frac{3m-1}{6(m-1)} \delta^2 - \frac{\pi}{4}\delta + \frac{1}{m-1} <0\\
  \Leftrightarrow \quad \delta^2 > \frac{6-3\pi\delta(m-1)/2}{3m-1} \quad &\Leftarrow \quad \delta \ge \frac{1}{m-1}
  \end{align*}
  
  Furthermore, observe
  \begin{align*}
  \frac{1}{m-1}\sin\theta + \theta\,\cos\theta > 0 \quad &\Leftrightarrow \quad \frac{1}{m-1}\cos\delta - \left( \frac{\pi}{2} + \delta \right)\,\sin\delta > 0 \\
  \Leftarrow \quad \frac{1 - \delta^2/2}{m-1} - \left( \frac{\pi}{2} + \delta \right)\delta > 0 \quad &\Leftrightarrow \quad  - \frac{2m-1}{2(m-1)} \delta^2 - \frac{\pi}{2}\delta + \frac{1}{m-1} >0\\
  \Leftrightarrow \quad \delta^2 < \frac{2-\pi\delta(m-1)}{2m-1} \quad & \Leftarrow \quad \delta \le \frac{1}{3(m-1)}% \qedhere
  \end{align*}
\end{proof}

To prove similar bounds for $\theta_{m,4}$ we need a technical auxiliary result.

\begin{Lem} \label{lem:g-bound}
  We have the following bound
  \begin{align*}
  g_m(\delta) := (3m-5) \cos^3\delta + (2m-2) \delta \, \sin\delta \cos^2\delta \le (3m-5) - (m-3)\delta^2\,.
  \end{align*}
\end{Lem}
\begin{proof}
  For $m=2$, the inequality reads
  \begin{align*}
  g_2(\delta) = \cos^3\delta + 2 \delta \, \sin\delta \cos^2\delta \le 1 + \delta^2\, .
  \end{align*}
  Inserting $\delta=0$ the inequality holds. Now consider the derivative
  \begin{align*}
  g'_2(\delta) = 2 \delta \, \cos\delta -\sin\delta \cos^2\delta - 6 \delta \, \sin^2\delta \cos\delta \le 2 \delta \, ,
  \end{align*}
  which is obviously true and thus proves the claim for $m=2$.
  
  For the induction step, we must show
  \begin{align*}
  g_0(\delta) := 3 \cos^3\delta + 2 \delta \, \sin\delta \cos^2\delta \le 3 - \delta^2\,.
  \end{align*}
  To this end, we will show the following sequence of inequalities
  \begin{align*}
  3 \cos^3\delta + 2 \delta \, \sin\delta \cos^2\delta \le 3 \cos^2\delta + \frac{\delta^2}{2} \cos^2\delta \le 3 \cos^2\delta + \frac{1}{2} \sin^2\delta \le 3 - \delta^2\,.
  \end{align*}
  For the first inequality, note that
  \begin{align*}
  && 3 \cos\delta + 2 \delta \, \sin\delta &\le 3 + \frac{\delta^2}{2}\\
  \quad \Leftarrow \quad 3 &\le 3 \quad \textnormal{and}& - \sin\delta + 2 \delta \, \cos\delta &\le \delta\\
  \quad \Leftarrow \quad 0 &\le 0 \quad \textnormal{and}& \cos\delta - 2 \delta \, \sin\delta &\le 1
  \end{align*}
  where the left inequalities reflect the values for $\delta=0$ and we perform derivatives from row to row. This proves the first inequality. Next we note
  \begin{align*}
  3 \cos^2\delta + \frac{\delta^2}{2} \cos^2\delta &\le 3 \cos^2\delta + \frac{1}{2} \sin^2\delta\\
  \Leftrightarrow \quad \delta \cos\delta &\le \sin\delta \, ,
  \end{align*}
  which proves the second inequality. This last estimate has the important property that its second derivative is monotonically growing on $[0,\pi/2]$. Therefore, once $3 \cos^2\delta + \frac{1}{2} \sin^2\delta > 3 - \delta^2$ for some delta, it would also hold for all larger $\delta$, particularly $\delta = \pi/2$. Since
  \begin{align*}
  \frac{1}{2} \le 3 - \frac{\pi^2}{4}
  \end{align*}
  we have finally shown $g_0(\delta) \le 3 - \delta^2$ as desired.
\end{proof}

\begin{Lem}
  The position $\theta_{m,4}$ of the zero of $\theta \mapsto f_4(\theta,0)$ is bounded by
  \begin{align*}
  \frac{\pi}{2} + \frac{1}{3(m-1)} &\le \theta_{m,2} \le \theta_{m,4} \le \frac{\pi}{2} + \frac{16}{\pi(m-3)} \, .
  \end{align*}
\end{Lem}

\begin{proof}
  To get a lower bound on the region of positive fourth derivative, we write again $\theta = \pi/2 + \delta$ and define
  \begin{align*}
  f (\delta) = \Big( (3m-9) - (3m-5) \cos^2\delta \Big) \cos\delta + \Big( (3m-9) - (2m-2) \cos^2\delta \Big) \left(\frac{\pi}{2} + \delta \right) \, \sin\delta \, .
  \end{align*}
  We can immediately see
  \begin{align*}
  f (\delta) &\ge (3m-9) + \frac{\pi}{2} (3m-9)\sin\delta - (3m-5) \cos^3\delta - (2m-2) \left(\frac{\pi}{2} + \delta \right) \, \sin\delta \cos^2\delta\\
  &\ge (3m-9) + \frac{\pi(m-7)}{2} \sin\delta - (3m-5) \cos^3\delta - (2m-2) \delta \, \sin\delta \cos^2\delta\\
  &\ge (3m-9) + \frac{\pi(m-7)}{4} \delta - (3m-5) \cos^3\delta - (2m-2) \delta \, \sin\delta \cos^2\delta \, .
  \end{align*}
  Using Lemma \ref{lem:g-bound} we can write
  \begin{align*}
  f (\delta) &\ge (3m-9) + \frac{\pi(m-7)}{4} \delta - (3m-5) + (m-3)\delta^2 \, .
  \end{align*}
  and plugging in $\theta_{m,4} - \pi/2 = \frac{16}{\pi(m-3)}$, we get
  \begin{align*}
  f (\theta_{m,4} - \pi/2) &\ge - 4(m-3) + 4(m-7) + \frac{256}{\pi^2} = \frac{256}{\pi^2} - 16 > 0 \, .
  \end{align*}
  
  Next, we aim to prove $\theta_{m,2} \le \theta_{m,4}$. Using
  \begin{align*}
  f (\delta) = - \Big( 4 - (3m-5) \sin^2\delta \Big) \cos\delta + \Big( (m-7) + (2m-2) \sin^2\delta \Big) \left(\frac{\pi}{2} + \delta \right) \, \sin\delta
  \end{align*}
  it suffices to show for $\delta_{m,2} := \theta_{m,2} - \frac{\pi}{2}$
  \begin{align*}
  -f(\delta_{m,2}) < 0 \quad &\Rightarrow \quad \cos \delta_{m,2} - (m-1) \left(\frac{\pi}{2} + \delta_{m,2} \right) \, \sin\delta_{m,2} < 0 \, ,
  \end{align*}
  which is equivalent to
  \begin{align*}
  \cos\delta_{m,2} &< \frac{ m-7 + 2(m-1) \sin^2\delta_{m,2} } {4 - (3m-5) \sin^2\delta_{m,2}} \left(\frac{\pi}{2} + \delta_{m,2} \right) \, \sin\delta_{m,2}\\
  \Rightarrow \quad \cos \delta_{m,2} &< (m-1) \left(\frac{\pi}{2} + \delta_{m,2} \right) \, \sin\delta_{m,2} \, .
  \end{align*}
  This is satisfied if
  \begin{align*}
  \frac{ m-7 + 2(m-1) \sin^2\delta_{m,2} }{4 - (3m-5) \sin^2\delta_{m,2}} &\le (m-1)\\
  \Leftrightarrow \qquad \sin^2\delta_{m,2} &\le \frac{m+1}{(m-1)^2}\\
  \Leftarrow \qquad \delta_{m,2} &\le \frac{\sqrt{m+1}}{m-1} \, .
  \end{align*}
  Since we have shown the upper bound $\delta_{m,2} \le \frac{1}{m-1}$, this always holds which proves the claim.
\end{proof}

\subsection{Hemisphere Model with Hole at the Cut Locus}

The Fr\'echet function of any rotation symmetric random variable on $\mathbb{S}^m$ can be expressed as
\begin{align*}
F (\psi) := \int \frac{1}{g (\theta)} F_\theta(\theta,\psi) \, d \mathbb{P}(\theta)
\end{align*}
by using a probability measure $d \mathbb{P}(\theta)$ supported on $[0,\pi]$, satisfying $\int g (\theta) \, d \mathbb{P}(\theta) = 1$. The results of the previous subsection can then be used to calculate the second and fourth derivative of the Fr\'echet function at the north pole.

We now present a family of random variables which exhibit a local Fr\'echet mean at the north pole while the range of the random variable has a hole containing the south pole. Consider a random variable $X$ distributed on the $m$-dimensional unit sphere $\mathbb{S}^m$ ($m\geq 4$) that is uniformly distributed on $\mathbb{L}_{m,\beta}$ with total mass $0<\alpha<1$ and assuming $\mu$ with probability $1-\alpha$. Then we have the \emph{Fr\'echet function}
\begin{align*}
F (\alpha, \beta, \psi) :=& (1- \alpha) \psi^2 + \alpha g(\beta) \int_{\pi/2}^{\pi - \beta} \sin\theta \int_0^\pi s^{m-2} \, a^2 \, d\phi \, d\theta \\
=& \psi^2 + \alpha g(\beta) \int_{\pi/2}^{\pi - \beta} \sin\theta \int_0^\pi s^{m-2} \, (a^2 - \psi^2) \, d\phi \, d\theta \\
g(\beta) :=& \left( \int_{\pi/2}^{\pi - \beta} \sin\theta \int_0^\pi s^{m-2} \, d\phi \, d\theta \right)^{-1} \, .
\end{align*}

In order to see that this family of random random variables can exhibit smeariness for suitably chosen values of $\alpha$ and $\beta$, we now calculate its second and fourth derivative at the north pole. The second derivative is
\begin{align*}
\frac{\partial^2F}{\partial\psi^2} (\alpha, \beta, 0) &= 2 (1- \alpha) + 2 \alpha g(\beta) \int_{\pi/2}^{\pi - \beta} f_2(\theta,0) \, d\theta\\
&= 2 (1- \alpha) + 2 \alpha I_{m} g(\beta) \frac{1}{m-1} \int_{\pi/2}^{\pi - \beta} \frac{d}{d\theta} \left( \theta\,\sin^{m-1}\theta \right) \, d\theta\\
&= 2 (1- \alpha) + 2 \alpha I_{m} g(\beta) \frac{1}{m-1} \left( \theta\,\sin^{m-1}\theta \right) \Big|_{\pi/2}^{\pi - \beta}\\
&= 2 (1- \alpha) + 2 \alpha I_{m} g(\beta) \frac{1}{m-1} \left( (\pi - \beta)\,\sin^{m-1}(\pi - \beta) - \pi/2 \right) \, .
\end{align*}

The fourth derivative is
\begin{align*}
&\frac{\partial^4F}{\partial\psi^4}(\alpha, \beta, 0) = 2 \alpha g(\beta) \int_{\pi/2}^{\pi - \beta} f_4(\theta,0) \, d\theta\\
=& \frac{2 \alpha I_{m} g(\beta)}{m+2} \int_{\pi/2}^{\pi - \beta} \sin^{m-3}\theta \Big( (3m-6) - (3m-3) \sin^2\theta \Big) \, d\theta\\
&- \frac{2 \alpha I_{m} g(\beta)}{m+2} \int_{\pi/2}^{\pi - \beta} \frac{d}{d\theta} \Big(3\theta\sin^{m-3}\theta - 2 \theta \sin^{m-1}\theta\Big) \, d\theta\\
=& \frac{2 \alpha I_{m} g(\beta)}{m+2} \int_{\pi/2}^{\pi - \beta} \Big( (3m-6) \sin^{m-3}\theta - (3m-3) \sin^{m-1}\theta \Big) \, d\theta\\
&+ \frac{2 \alpha I_{m} g(\beta)}{m+2} \Big( \frac{\pi}{2} - (\pi - \beta)\,\Big(3\sin^{m-3}(\pi - \beta) - 2 \sin^{m-1}(\pi - \beta)\Big) \Big)\\
=& \frac{2 \alpha I_{m} g(\beta)}{m+2} \Big( 3 \sin^{m-2}\beta \cos(\pi - \beta) - 3(\pi - \beta)\,\sin^{m-3}\beta\\
&+ 2(\pi - \beta)\, \sin^{m-1}\beta + \frac{\pi}{2} \Big)\, .
\end{align*}

Recall that due to Lemma \ref{lem:derivatives-integral} the result for the second derivative holds for $m \ge 2$ if $\beta > 0$ and for $m \ge 3$ otherwise. The result for the fourth derivative holds for $m \ge 4$ if $\beta > 0$ and for $m \ge 5$ otherwise.

In the following, for a fixed value of $\beta$ the value of $\alpha$ for which the second derivative at $\psi = 0$ vanishes will be denoted by $\alpha_\beta$.

\begin{Lem}\label{lem:bounds-beta2}
  For every $\beta < \beta_{m,2}$ there is an $0 \le \alpha_\beta \le 1$ such that $\frac{\partial^2F}{\partial\psi^2} (\alpha_\beta, \beta, 0) = 0$. This $\beta_{m,2}$ satisfies the bounds
  \begin{align*}
  \frac{\pi}{2} - \frac{6}{\pi(m-1)} &\le \beta_{m,2} \le \frac{\pi}{2} - \frac{1}{2(m-1)}
  \end{align*}
\end{Lem}

\begin{proof}
  From the calculated second derivative we get
  \begin{align*}
  &1/ \alpha_\beta = 1 + \frac{1}{m} \left( \int_{\pi/2}^{\pi - \beta} \sin^{m-1}\theta \, d\theta \right)^{-1}  \left( \pi/2 - (\pi - \beta)\,\sin^{m-1}\beta \right) \, .
  \end{align*}
  To have a valid random variable, we need $0 \le \alpha_\beta \le 1$. Defining the function
  \begin{align*}
  b_{m,2} (\beta) := \pi/2 - (\pi - \beta)\,\sin^{m-1}\beta
  \end{align*}
  the condition $0 \le \alpha_\beta \le 1$ is equivalent to $b_{m,2} (\beta) \ge 0$. Let $\beta_{m,2}$ be the first zero of $b_{m,2}$. To see that $\beta_{m,2} < \pi/2$ note that
  \begin{align*}
  b_{m,2} (\beta) &\le \frac{\pi}{2} - (\pi - \beta) \left( 1 - \frac{m-1}{2} \left(\frac{\pi}{2} - \beta \right)^2 \right)\\
  &= -\left(\frac{\pi}{2} - \beta \right) + (\pi - \beta) \left( \frac{m-1}{2} \left(\frac{\pi}{2} - \beta \right)^2 \right)\\
  &= \frac{m-1}{2} \left(\frac{\pi}{2} - \beta \right)^3 + \frac{\pi(m-1)}{4} \left(\frac{\pi}{2} - \beta \right)^2 - \left(\frac{\pi}{2} - \beta \right)\\
  \end{align*}
  Plugging $\beta = \frac{\pi}{2} - \frac{1}{2(m-1)}$ into the right hand side, we get for $m \ge 2$
  \begin{align*}
  b_{m,2} \left(\frac{\pi}{2} - \frac{1}{2(m-1)}\right) &\le \frac{1 - (8-\pi)(m-1)}{16(m-1)^2} \le 0\\
  \Rightarrow \qquad \beta_{m,2} &\le \frac{\pi}{2} - \frac{1}{2(m-1)} \, .
  \end{align*}
  
  Furthermore, note that using $\delta=\frac{\pi}{2} - \beta$,
  \begin{align*}
  &b_{m,2} = \frac{\pi}{2} - \left(\frac{\pi}{2} + \delta\right) \cos^{m-1}\delta \ge 0 \quad \Leftrightarrow \quad \cos\delta \le \left(1 + \frac{2\delta}{\pi}\right)^{-\frac{1}{m-1}}\\
  & \Leftarrow \quad 1 - \frac{\delta^2}{3} \le 1 - \frac{2\delta}{\pi(m-1)} \quad \Leftarrow \quad \delta \le \frac{6}{\pi(m-1)}
  \end{align*}
  and therefore
  \begin{align*}
  \beta_{m,2} &\ge \frac{\pi}{2} - \frac{6}{\pi(m-1)} \, . %\qedhere
  \end{align*}
\end{proof}

\begin{Lem}\label{lem:bounds-beta4}
  There is a $\beta_{m,4}$ such that every $\beta < \beta_{m,4}$ satisfies $\frac{\partial^4F}{\partial\psi^4} (\alpha_\beta, \beta, 0) > 0$. This $\beta_{m,4}$ satisfies the bounds
  \begin{align*}
  \frac{\pi}{2} - \frac{6(6+\pi)}{\pi(m-3)} &\le \beta_{m,4} \le \frac{\pi}{2} - \frac{1}{m-3}\\
  \frac{\pi}{2} - \frac{6(6+\pi)}{\pi(m-3)} &\le \beta_{m,4} \le \beta_{m,2} \le \frac{\pi}{2} - \frac{1}{2(m-1)}
  \end{align*}
\end{Lem}

\begin{proof}
  Using
  \begin{align*}
  b_{m,4} (\beta) := \pi/2 + 2(\pi - \beta)\, \sin^{m-1}\beta - 3 \cos\beta \sin^{m-2}\beta - 3(\pi - \beta)\,\sin^{m-3}\beta
  \end{align*}
  we can give the necessary condition $b_{m,4} (\beta) > 0$ for a local minimum of the Fr\'echet function at $\psi=0$. From this relation we can determine a minimal $\beta_{m,4}$ such that $b_{m,4} (\beta_{m,4}) \ge 0$ for every dimension $m \ge 4$ giving a dimension dependent maximal hole size. Note that
  \begin{align*}
  &b_{m,2} (\beta) - b_{m,4} (\beta)\\
  =& -3(\pi - \beta)\, \sin^{m-1}\beta + 3 \cos\beta \sin^{m-2}\beta + 3(\pi - \beta)\,\sin^{m-3}\beta\\
  =& 3 \cos\beta \sin^{m-2}\beta  + 3(\pi - \beta)\, \cos^2\beta \sin^{m-3}\beta \ge 0 \, ,
  \end{align*}
  which implies $\beta_{m,4} \le \beta_{m,2}$.
  
  Let $\delta = \frac{\pi}{2} - \beta$ and use $1 - \delta^2/2 \le \cos\delta \le 1 - \delta^2/3$ and $\delta/2 \le \sin\delta \le \delta$ , which hold on $[0,\pi/2]$, then, assuming $m \ge 3$
  \begin{align*}
  b_{m,4} (\beta) &= \frac{\pi}{2} + 2(\pi - \beta)\, \sin^{m-1}\beta - 3 \cos\beta \sin^{m-2}\beta - 3(\pi - \beta)\,\sin^{m-3}\beta\\
  &= \frac{\pi}{2} + 2 \left( \frac{\pi}{2} + \delta \right)\, \cos^{m-1}\delta - 3 \sin\delta \cos^{m-2}\delta - 3\left( \frac{\pi}{2} + \delta \right)\,\cos^{m-3}\delta\\
  &= \frac{\pi}{2} - (1 + 2 \sin^2 \delta) \left( \frac{\pi}{2} + \delta \right)\, \cos^{m-3}\delta - 3 \sin\delta \cos^{m-2}\delta\\
  &\le \frac{\pi}{2} - \left( \frac{\pi}{2} + \delta \right)\,\left( 1 - \frac{m-3}{2} \delta^2 \right) - \frac{3}{2} \delta \left( 1 - \frac{m-2}{2} \delta^2 \right)\\
  &= -\frac{5}{2} \delta + \frac{\pi(m-3)}{4} \delta^2 + \frac{5(m-3) + 3}{4} \delta^3 \, .
  \end{align*}
  Now, plugging in $\delta = \frac{1}{2(m-3)}$ we get for $m \ge 4$
  \begin{align*}
  b_{m,4} (\beta) &\le -\frac{5}{4(m-3)} + \frac{\pi}{16(m-3)} + \frac{(5+\pi)(m-3) + 3}{32(m-3)^3}\\
  &= \frac{(-40 + 2\pi)(m-3)^2 + (5+\pi)(m-3) + 3}{32(m-3)^3} \le 0 \, .
  \end{align*}
  From this, we get the upper bound
  \begin{align*}
  \beta_{m,4} \le \frac{\pi}{2} - \frac{1}{2(m-3)} \, .
  \end{align*}
  
  Analogously, we show the lower bound, by first noting
  \begin{align*}
  b_{m,4} (\beta) &= \frac{\pi}{2} - (1 + 2 \sin^2 \delta) \left( \frac{\pi}{2} + \delta \right)\, \cos^{m-3}\delta - 3 \sin\delta \cos^{m-2}\delta\\
  &= \frac{\pi}{2} - \left( \frac{\pi}{2} + \delta + 2 \delta\, \sin^2 \delta + \pi \sin^2 \delta + 3 \sin\delta \cos\delta \right) \cos^{m-3}\delta\\
  &\ge \frac{\pi}{2} - \left( \frac{\pi}{2} + (6+\pi) \delta \right) \cos^{m-3}\delta
  \end{align*}
  and then calculating
  \begin{align*}
  &\frac{\pi}{2} - \left(\frac{\pi}{2} + (6+\pi) \delta\right) \cos^{m-3}\delta \ge 0\\
  &\Leftrightarrow \quad \cos\delta \le \left(1 + \frac{2(6+\pi)\delta}{\pi}\right)^{-\frac{1}{m-3}} \quad \Leftarrow \quad 1 - \frac{\delta^2}{3} \le 1 - \frac{2(6+\pi)\delta}{\pi(m-3)} \quad \Leftarrow \quad \delta \le \frac{6(6+\pi)}{\pi(m-3)}
  \end{align*}
  which establishes the lower bound
  \begin{align*}
  \beta_{m,4} &\ge \frac{\pi}{2} - \frac{6(6+\pi)}{\pi(m-3)} \, . %\qedhere
  \end{align*}
\end{proof}

Numerically determined values of $\beta_{m,2}$ and $\beta_{m,4}$ along with the lower bounds from Lemmas \ref{lem:bounds-beta2} and \ref{lem:bounds-beta4} are displayed in Figure \ref{fig:betas}.

\begin{figure}[h!]
  \centering
  \includegraphics[width=0.8\textwidth]{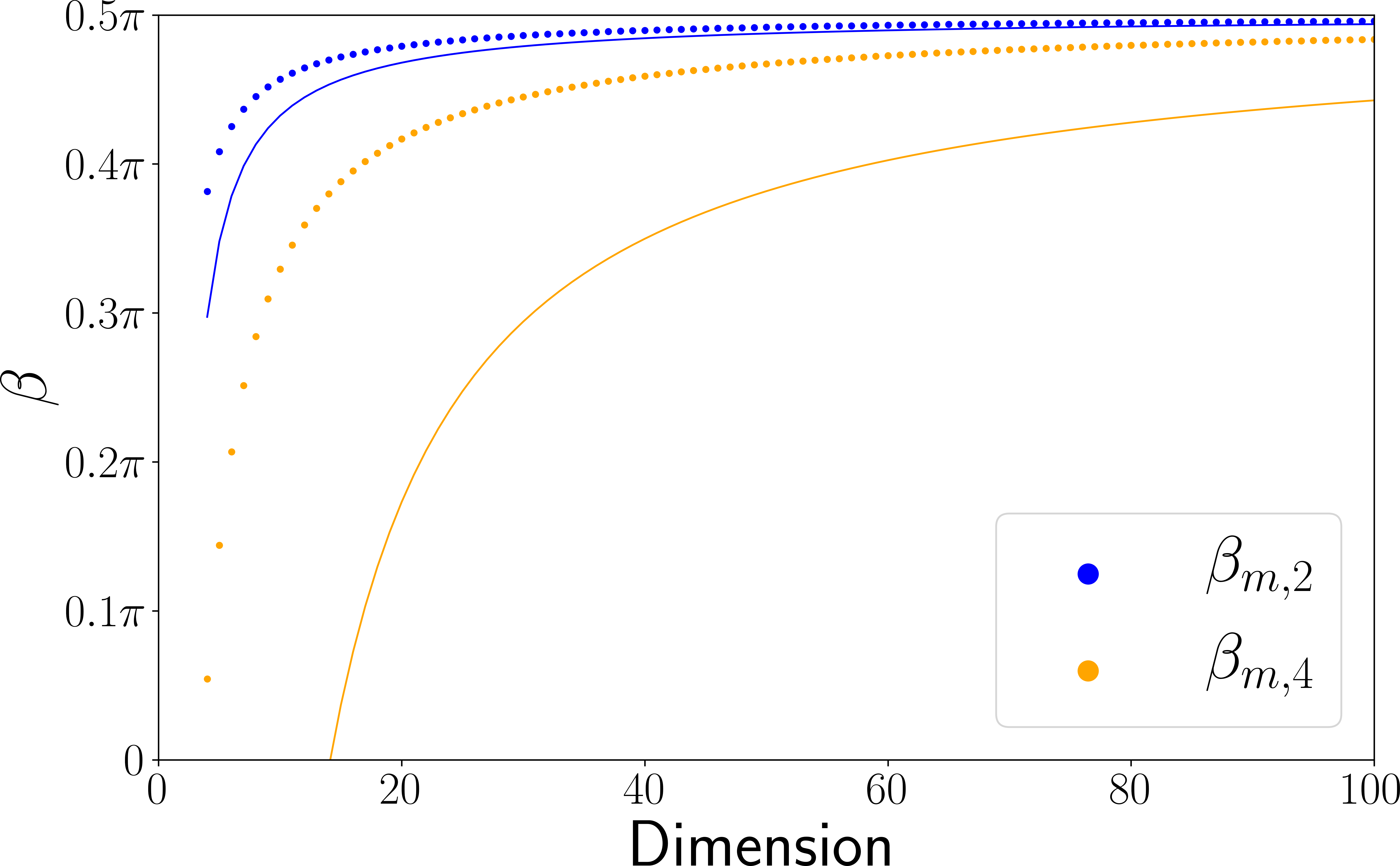}
  \caption{\it Numerically determined values for $\beta_{m,2}$ and $\beta_{m,4}$ which bound the radius of the hole from above for $m \le 100$. One can clearly see that the values approach $\pi/2$ from below. \label{fig:betas}}
\end{figure}

%Using the upper bound for $\beta_{m,2}$, we can derive a lower bound for $\alpha_\beta$, which is stronger than $\alpha_\beta \ge 0$, namely
%\begin{align*}
%1/ \alpha_\beta \le 1 + \frac{24\pi(m-1)^2}{m(24(m-1)-1)} \, .
%\end{align*}
An $\alpha_\beta$, such that the Hessian vanishes, exists for ${\left( \pi/2 - (\pi - \beta)\,\sin^{m-1}\beta \right) > 0}$, which can indeed be satisfied, as evidenced by the lower bound $\beta_{m,2}$ determined here. Note that this result is valid for all $m \ge 2$, but we only get a positive fourth derivative and thus a local minimum in the case of vanishing Hessian for dimension $m \ge 4$.

To show that the local minimum at $\psi = 0$ is indeed a global minimum at least for some $\beta > 0$, we use a Lipschitz argument as follows. Recall that for $\beta = 0$
\begin{align*}
\frac{\partial^2F}{\partial\psi^2} &>0 & \textnormal{for } \psi &\in (0,\pi) &\frac{\partial^3F}{\partial\psi^3} &>0 & \textnormal{for } \psi &\in (0,\pi)\\
\frac{\partial^4F}{\partial\psi^4} &>0 & \textnormal{for } \psi &\in [0,\pi/2) \, .
\end{align*}

To show that we have a global minimum at $\psi = 0$, we need for every $\psi \in (0,\pi]$ that $\frac{\partial^2F}{\partial\psi^2} >0$. In order to show this, we prove the following Lipschitz conditions, where $\alpha_i$ denotes the $\alpha_\beta$ corresponding to $\beta_i$.
\begin{Lem} \label{lem:smeary-lipschitz}
  There are dimension dependent constants $L_2$, $L_3$ and $L_4$ such that
  \begin{align*}
  \left| \frac{\partial^2F}{\partial\psi^2} (\alpha_1, \beta_1, \psi) - \frac{\partial^2F}{\partial\psi^2} (\alpha_2, \beta_2, \psi) \right| &\le L_2 |\beta_1 - \beta_2 |  & \textnormal{for } m &\ge 3 \\
  \left| \frac{\partial^3F}{\partial\psi^3} (\alpha_1, \beta_1, \psi) - \frac{\partial^3F}{\partial\psi^3} (\alpha_2, \beta_2, \psi) \right| &\le L_3 |\beta_1 - \beta_2 |  & \textnormal{for } m &\ge 4 \\
  \left| \frac{\partial^4F}{\partial\psi^4} (\alpha_1, \beta_1, \psi) - \frac{\partial^4F}{\partial\psi^4} (\alpha_2, \beta_2, \psi) \right| &\le L_4 |\beta_1 - \beta_2 |  & \textnormal{for } m &\ge 5 \, . \\
  \end{align*}
\end{Lem}

\begin{proof}
  Note that
  \begin{align*}
  L_j \geq \max_{\beta \in [0,\beta_{m,4})\, , \, \psi \in [0,\pi]} \limits \left| \frac{\partial^{j+1}F}{\partial\beta \partial\psi^j} (\alpha, \beta, \psi) \right|
  \end{align*}
  are valid Lipschitz constants. Thus we note for $j = 2, 3, 4$
  \begin{align*}
  \frac{\partial^{j+1}F}{\partial \beta \partial\psi^j} =& 2 \frac{d}{d \beta} \left( (1- \alpha_\beta) \delta_{j2} +  \alpha_\beta g(\beta) \int_{\pi/2}^{\pi - \beta} f_j(\theta,\psi) \, d\theta \right)\\
  =& - 2 \frac{\partial\alpha}{\partial\beta} \delta_{j2} + 2g \frac{\partial\alpha}{\partial\beta} \int_{\pi/2}^{\pi - \beta} f_j(\theta,\psi) \, d\theta\\
  &+ 2 \alpha \frac{\partial g}{\partial\beta} \int_{\pi/2}^{\pi - \beta} f_j(\theta,\psi) \, d\theta - 2 \alpha g f_j(\pi - \beta,\psi)
  \end{align*}
  We know $|\alpha | = \alpha < 1$ and $|g| = g \le g(0)$. $\left| \frac{\partial\alpha}{\partial\beta} \right|$ and $ \left| \frac{\partial g}{\partial\beta} \right|$ can also trivially be bounded, since $\beta_{m,4} < \pi/2$. So only $f_j(\theta,\psi)$ and their $\theta$ integrals remain to be bounded. Since the numerators can all be easily bounded, the only problem is to bound the denominators under the integrals. Using the boundedness shown in Lemma \ref{lem:derivatives-integral} we see that we need $m \ge 5$ for these bounds to be finite.
  
  Collecting all the estimates, we get the desired Lipschitz constants.
\end{proof}

Using the Lipschitz constants, we can now show that the local minima are global for suitably small $\beta > 0$. Since we need these bounds to hold on the full range $\psi \in [0, \pi]$, we cannot improve dimension as simply as in the case of the derivatives at $\psi = 0$. Note that the estimates in Equation \eqref{eq:smeary-estimates} which these calculations rely on are very generous and might be improved by a more careful treatment.

\subsection{Proof of Theorem 4.2}

Consider a probability measure with a point mass at the north pole with weight $1-\alpha$ and a uniform distribution with weight $\alpha$ on the $\mathbb{S}^{m-1}$ at $\theta = \theta_*$. Since the contribution of the second term to the second derivative of the Fr\'echet function at the north pole is negative and the contribution to the fourth derivative is positive, there is an $\alpha_0 > 0$ such that the Hessian of the Fr\'echet function at the north pole vanishes and the fourth derivative is positive, such that there is a local Fr\'echet mean at the north pole. For general $\alpha \in [0,1]$ the Hessian of the Fr\'echet function at the north pole is given by
\begin{align*}
H := 2(1 - \alpha) \textnormal{Id}_m - 2\frac{\alpha(1- \alpha_0)}{\alpha_0} \textnormal{Id}_m = 2 \frac{\alpha_0 - \alpha}{\alpha_0} \textnormal{Id}_m \, .
\end{align*}

Since $\textnormal{Cov}[\grad\rho(0,X)] = \frac{4 \alpha}{m} \theta_*^2 \textnormal{Id}_m$, we get
\begin{align*}
\lim_{n\to\infty} \limits n \textnormal{Var}[\widehat{\mu}_n] = \frac{\alpha_0^2 \alpha \theta_*^2}{(\alpha_0 - \alpha)^2} = \frac{\alpha_0^2}{(\alpha_0 - \alpha)^2} \textnormal{Var}[X]
\end{align*}

Since $\alpha < \alpha^0$ can be freely chosen, the claim follows.

\bibliographystyle{Chicago}
\bibliography{bibliography}

\end{document}